\documentclass[11pt, reqno]{amsart}
\title{Gaussian limits of lattice Higgs models with complete symmetry breaking}
\date{}

\author[Frederick Rajasekaran]{Frederick Rajasekaran}
\author[Oren Yakir]{Oren Yakir}
\author[Yanxin Zhou]{Yanxin Zhou}
\address{Frederick Rajasekaran \newline Department of Mathematics, Stanford University, USA.}
\email{fredr@stanford.edu \vspace{0.2cm}}
\address{Oren Yakir \newline Department of Mathematics, Massachusetts Institute of Technology, USA.}
\email{oren.yakir@gmail.com \vspace{0.2cm}}
\address{Yanxin Zhou \newline Department of Statistics, Stanford University, USA.}
\email{yanxin01@stanford.edu}
    
\usepackage{amsmath,amsfonts,amssymb,amsthm,graphics,enumerate,mathtools,float,tikz}
\usepackage{hyperref,dsfont}
\hypersetup{
    colorlinks=true,
    linkcolor=blue,
    citecolor=red,
    urlcolor=blue,
    pdfborder={0 0 0}
}

\usepackage{graphicx,xcolor,bbm,fullpage}
\usepackage{geometry}
\newgeometry{vmargin={31 mm}, hmargin={26mm,26mm}}
\usepackage{booktabs}

\newtheorem{thm}{Theorem}
\newtheorem{claim}[thm]{Claim}
\newtheorem{lemma}[thm]{Lemma}

\newtheorem{proposition}[thm]{Proposition}

\newtheorem{definition}{Definition}

\numberwithin{equation}{section}

\theoremstyle{definition}
\newtheorem{remark}[thm]{Remark}

\def\cX{\mathcal{X}}

\def\cP{\mathcal{P}}

\def\cE{\mathcal{E}}

\def\cA{\mathcal{A}}

\def\eps{{\varepsilon}}

\newcommand{\bR}{\mathbb R}

\newcommand{\bZ}{\mathbb Z}

\newcommand{\bT}{\mathbb T}
\newcommand{\bN}{\mathbb N}

\newcommand{\bP}{\mathbb P}
\newcommand{\bS}{\mathbb S}
\newcommand{\bE}{\mathbb E}

\DeclareMathOperator{\cov}{Cov}

\DeclareMathOperator{\Tr}{Tr}

\newcommand{\Var}{{\sf Var}}
\newcommand{\Log}{{\sf Log}}

\begin{document}

\maketitle

\begin{abstract}
Given any compact connected matrix Lie group $G$ and any lattice dimension $d\ge 2$, we construct a massive Gaussian scaling limit for the $G$-valued lattice Yang-Mills-Higgs theory in the ``complete breakdown of symmetry" regime. This limit arises as the lattice spacing tends to zero and the (inverse) gauge coupling constant tends to infinity sufficiently fast, causing the theory to ``abelianize" and yield a Gaussian limit. This complements a recent work by Chatterjee~\cite{Chatterjee-HiggsMechanism}, 
% when uploading to the arXiv, make sure to remove the cite command
which obtained a similar scaling limit in the special case $G= SU(2)$.
\vspace{1em}


\end{abstract}

\section{Introduction}
\noindent
One of the big open problems in mathematical physics is to rigorously construct non-Gaussian scaling limits of lattice gauge theories. This is the essence of the Yang-Mills existence and mass gap problem, one of the Clay Millennium Prize problems~\cite{Jaffe-Witten}. In this paper, we identify a regime in which the scaling limit is Gaussian. Specifically, we study Yang-Mills gauge fields coupled to a Higgs field with ``complete symmetry breaking". That is, when the Higgs field takes values in the gauge group and the field acts on it via the trivial representation; see Section~\ref{subsec:lattice_YMH_model} for the formal definition, and Section~\ref{subsec:complete_breakdown_of_symmetry} for additional background. In Theorem~\ref{thm:main_result} below, we show that as the lattice spacing tends to zero and the inverse gauge coupling diverges, a lifting of the lattice gauge field converges to the Proca field---a random, Lie-algebra valued, generalized 1-form which is both Gaussian and massive. A version of this result was recently obtained by Chatterjee~\cite{Chatterjee-HiggsMechanism} for the gauge group SU(2), and we show that the construction can be preformed, after some non-trivial modifications, with an arbitrary compact matrix Lie group acting as the gauge group of the theory. 

\subsection{The lattice Yang-Mills-Higgs model}
\label{subsec:lattice_YMH_model}
Let $G$ be a compact connected matrix Lie group, such that $G\subset U(N)$\footnote{Here, $U(N)$ is the group of $N\times N$ unitary matrices. } for some $N\ge 1$. Throughout, $G$ is referred to as the \textbf{gauge group} of the theory. We denote by $\mathfrak{g}$ the corresponding Lie algebra, that is, the tangent space at the identity element $I\in G$, and let $n = \text{dim}(G)$. For two $N\times N$ matrices $A,B$ with complex entries, we recall their Hilbert-Schmidt inner product is given by
\begin{equation}
    \label{eq:HS_inner_product}
    \langle A,B\rangle = \Tr(A B^\ast) \, , \qquad \| A \|^2 = \langle A,A \rangle \, .
\end{equation}
For $d\ge 2$ we set $\bT_{L}^d = (\bZ/L\bZ)^d$ to be the $d$-dimensional discrete torus with side length $L\ge 1$. We denote by $V,E,P$ the set of vertices, edges, and plaquettes in the underlying graph that we consider. In this paper, we study the \textbf{lattice Yang-Mills-Higgs model} in the \textbf{complete breakdown of symmetry regime}, as described in~\cite[Section~IV]{Fradkin-Shenker} or~\cite[Chapter~3]{Seiler-book}. This is a Gibbs measure with periodic boundary conditions for sampling a field $U = (U_e)_{e\in E(\bT_L^d)}$ taking values in the gauge group $G$, whose action is given by
\begin{equation} \label{eq:YMH_action_with_mass_m}
    \mathcal{H}(U) = \frac{1}{2} \sum_{p\in P(\bT_{L}^d)} \| I - ({\sf d}U)_p \|^2 + \frac{m}{2}\sum_{e\in E(\bT_L^d)} \| I - U_e\|^2 \, .
\end{equation}
Here, $m>0$ is a parameter called the \textbf{mass} and, for a plaquette $p$ bounded by four positively oriented edges $e_1,e_2,e_3,e_4$ (see Figure~\ref{figure:plaquette}), we have 
\[
({\sf d}U)_p  = U_{e_1}U_{e_2} U_{e_3}^{-1} U_{e_4}^{-1} \, .
\]
\begin{figure}
        \label{figure:plaquette}
		 	\begin{center}	\scalebox{0.8}{
            \begin{tikzpicture}[>=latex, thick]

            % Define coordinates for the corners of the square
            \coordinate (A) at (0,0);
            \coordinate (B) at (4,0);
            \coordinate (C) at (4,4);
            \coordinate (D) at (0,4);
            
            % Draw the directed edges
            \draw[->] (A) -- (B) node[midway, below=3pt] {$e_1$};
            \draw[->] (B) -- (C) node[midway, right=3pt] {$e_2$};
            \draw[->] (D) -- (C) node[midway, above=3pt] {$e_3$};
            \draw[->] (A) -- (D) node[midway, left=3pt] {$e_4$};

\end{tikzpicture}
            }
		 	\end{center}
            \caption{A plaquette bounded by the positively oriented edges $e_1,e_2,e_3,e_4$.}
\end{figure}
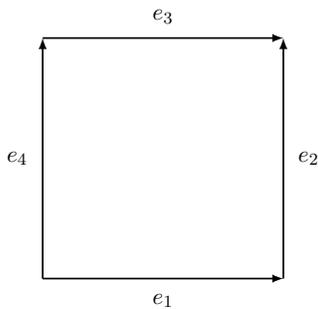
The Yang-Mills-Higgs lattice model is the probability measure $\nu_{\beta,m}^L$ given by
\begin{equation} \label{eq:def_of_YMH_measure_on_torus}
    {\rm d}\nu_{\beta,m}^L = \frac{1}{Z_{\beta,m,L}^{\sf YMH}}  e^{-\beta\mathcal{H}(U)} \prod_{e\in E(\bT_L^d)} {\rm d} \mu(U_e) \, ,
\end{equation}
where $\mu$ is the Haar measure on $G$, $Z_{\beta,m,L}^{\sf YMH}$ is the partition function, and $\beta>0$ is the inverse \textbf{gauge coupling}.\footnote{This model is sometimes stated in the literature with two parameters: $\beta$, the inverse gauge coupling, and $\kappa$, the Higgs length (see, e.g.\ \cite{Chatterjee-HiggsMechanism, forsstrom2026phasetransitions, Shen-Zhu-Zhu}). For our purposes, we prefer the parametrization where $m = \kappa/\beta$.} Since $G$ is compact, existence of weak limits of $\nu_{\beta,m}^L$ as $L\to \infty$ are guaranteed, and the resulting limits are invariant under translations of $\bZ^d$. With a slight abuse of notation, we will denote any such infinite-volume limit by $\nu_{\beta,m}$.  

\subsection{Logarithmic coordinates}
\label{subsec:logarithmic_coordinates}
    Recall that our gauge group $G$ is a compact connected matrix Lie group with Haar measure $\mu$, which we always normalize so that $\mu(G) = 1$. The Lie algebra $\mathfrak{g}$ is defined as the tangent space of $G$ at the identity element $I\in G$. This makes $\mathfrak{g}$ into a matrix vector space which is isomorphic to $\bR^n$, where $n=\dim(G)$. For our result, we will need to lift configurations from $G^{E(\bZ^d)}$ onto a configuration from $\mathfrak{g}^{E(\bZ^d)}$. We will do so via a local chart of the identity element of $G$, known as the \textbf{logarithmic coordinates}. Our treatment here follows the excellent book~\cite{Hall-LieBook}.
    
    Given any matrix $X$, the exponential of $X$ is given by
    \begin{equation}
		\label{eq:exp_of_matrix}
		\exp(X) = \sum_{k=0}^{\infty} \frac{X^k}{k!} = I + X + \frac{X^2}{2} + \cdots 
	\end{equation}
    It is not hard to check (see~\cite[Chapter~3]{Hall-LieBook}) that $\exp$ maps $\mathfrak{g}$ onto $G$ smoothly, with $\exp(0) = I$. Furthermore, the differential of $\exp$ at $0\in \mathfrak{g}$ is the identity map. Therefore, by the Inverse Function Theorem, we know that $\exp$ restricted to some neighborhood $\sf{U}$ containing 0 is a local diffeomorphism onto $\exp({\sf U}) \subset G$, which is a neighborhood of the identity matrix. We denote the inverse of the exponential map on ${\sf V} = \exp({\sf U})$ by $\log = \exp^{-1}$. Clearly, each choice of basis on $\mathfrak{g}$ determines a coordinate system on ${\sf U}$, and these are exactly the logarithmic coordinates (sometimes also called \emph{normal coordinates}). Later on, we will fix a basis on $\mathfrak{g}$ which is orthonormal with respect to the Hilbert–Schmidt norm~\eqref{eq:HS_inner_product}, so the reader can keep that type of an example in mind. For matrices $U\in G$ such that $\|I-U\| \le 1/2$, the logarithm is explicitly given by
	\begin{equation}
		\label{eq:log_of_matrix}
		\log (U) = -\sum_{k\ge 1} \frac{(I-U)^k}{k} \, .
	\end{equation} 
    We also remark that the neighborhood ${\sf U}\subset G$ always contains a geodesic ball of sufficiently small radius around the identity matrix $I\in G$.

\subsection{The Proca field}
The scaling limit obatained in our main result is that of a particular Gaussian generalized 1-form, known as the \textbf{Proca field}.

    \subsubsection*{The Euclidean Proca field}
    We follow the presentation as given in~\cite[Section~2]{Chatterjee-HiggsMechanism}. Let $\mathcal{A}(\bR^d)$ denote the space of Schwartzian 1-forms on $\bR^d$. That is, an element $F\in \mathcal{A}(\bR^d)$ is formally given by 
    $$
    F=F_1 \, {\rm d}x_1 + \ldots + F_d \, {\rm d}x_d
    $$
    where $F_1,\ldots,F_d\in \mathcal{S}(\bR^d)$ are in the space of Schwartz functions\footnote{A function $g:\bR^d\to \bR$ is called \textbf{Schwartz} if, for all $\alpha_1,\ldots,\alpha_d,\beta_1,\ldots,\beta_d\in \bN$, we have 
    $$
    \sup_{x\in \bR^d} \Big|x_1^{\alpha_1} \cdots x_d^{\alpha_d} \cdot \partial_{x_1}^{\beta_1} \cdots \partial_{x_d}^{\beta_d} \,  g(x) \Big| < \infty
    $$ }. As specifying a 1-form in $\bR^d$ is equivalent to specifying a vector field $F:\bR^d \to \bR^d$, we will always identify an element of $\mathcal{A}(\bR^d)$ as a tuple $F=(F_1,\ldots,F_d)$ of Schwartz functions. We endow $\mathcal{A}(\bR^d)$ with the inner product
    \begin{equation}
        \label{eq:inner_product_on_1_forms}
        (F,G)  =  \sum_{j=1}^{d}  \int_{\bR^d} F_j(x) G_j(x)  \, {\rm d}x \, .
    \end{equation}
    For $m > 0$, we set 
    \begin{equation}
        \label{eq:def_of_K_m}
        K_m = (-\Delta + mI)^{-1} \, ,
    \end{equation}
    where $\Delta$ is the Laplace operator and $I$ is the identity. For $d\ge 2$ and $m>0$, $K_m$ is in fact a bijection from $\mathcal{S}(\bR^d)$ onto itself, and commutes with all derivative operators, see~\cite[Lemma~2.1]{Chatterjee-HiggsMechanism}. With~\eqref{eq:def_of_K_m}, we can now define an operator $R_m:\mathcal{A}(\bR^d) \to \mathcal{A}(\bR^d)$ via
    \begin{equation}
        \label{eq:def_of_R_m}
        R_m F = K_m \big( F - m^{-1} \nabla \cdot \textbf{div}(F)\big) \, ,
    \end{equation}
    or, in coordinates, as
    \[
       R_m F = \sum_{j=1}^{d} \Big(K_m F_j  - m^{-1} \sum_{i=1}^d \partial_j\partial_i K_m F_i \Big) \, {\rm d}x_j \, .
    \]
    It is not hard to see (see, for instance, \cite[Section~5]{Cao-Sheffield-FGF-overview}) that $R_m^{-1} = (mI + {\rm d}^\ast {\rm d})$, where ${\rm d}$ is the usual exterior derivative and ${\rm d}^\ast$ is the dual to the exterior derivative with respect to the inner product~\eqref{eq:inner_product_on_1_forms}.
    \begin{definition}
        \label{def:proca_field}
        For $m>0$ we define the \textbf{Euclidean Proca field} on $\bR^d$ to be the random generalized 1-form $\cX_m$, such that for any $F\in \mathcal{A}(\bR^d)$, $\cX_m(F)$ is a Gaussian random variable with mean zero and variance $(F,R_mF)$, where $R_m$ is given by~\eqref{eq:def_of_R_m}.
    \end{definition}
    \noindent
    Clearly, the Proca field is translation invariant, and it is also scale invariant with proper normalization (see~\cite[Section~2.3]{Chatterjee-HiggsMechanism}). What is more important for our application, is that the Euclidean Proca field with parameter $m>0$ is in fact \emph{massive}, i.e.\ has exponential decay of correlations, see \cite[Lemma~2.6]{Chatterjee-HiggsMechanism} and Lemma~\ref{lemma:lattice_Proca_is_massive} below. Further properties of the Euclidean Proca field can be found in \cite{Cao-Sheffield-FGF-overview, Chatterjee-HiggsMechanism, Yao1975ProcaField} and references therein.
    
    \subsubsection*{Lie algebra valued Proca field}

    We will denote by $\mathfrak{D}$ the set of all compactly supported, $C^\infty$-smooth functions from $\bR^d$ to $\mathfrak{g}^d$. Recall that $\mathfrak{g}$ is a vector space of dimension $n=\text{dim}(G)$. Let $V_1,\ldots,V_n\in \mathfrak{g}$ be an orthonormal basis with respect to the inner product~\eqref{eq:HS_inner_product}. For $F\in \mathfrak{D}$, that is $F=(F_1,\ldots,F_d)$ such that
    \[
    \qquad F_j : \bR^d \to \mathfrak{g} \, , \qquad 1\le j\le d \, ,
    \]
    we define $C^\infty$-smooth, compactly supported functions $F^\ell : \bR^d \to \bR^d$ via the formula
    \[
    F_j^\ell (x) = \langle F_j(x) , V_\ell \rangle \, .
    \]
    \begin{definition}
    \label{def:g_valued_proca_field}
        For $m>0$ we define the \textbf{$\mathfrak{g}$-valued Proca field} on $\bR^d$ to be the random generalized $\mathfrak{g}$-valued 1-form $\cX_{\mathfrak{g},m}$, such that for any $F\in \mathfrak{D}$, the random variable $\cX_{\mathfrak{g},m}(F)$ is a Gaussian random variable with mean zero and variance
        \[
        \sum_{\ell = 1}^n (F^\ell ,R_m F^\ell) \, ,
        \]
        where $R_m$ is given by~\eqref{eq:def_of_R_m}. 
    \end{definition}
    \noindent
    It is evident that Definition~\ref{def:g_valued_proca_field} does not depend on the choice of the basis $\{V_1,\ldots,V_n\}$. In fact, this is exactly the Gaussian generalized function one obtains by considering $n$ independent Euclidean Proca fields $P_1,\ldots,P_n$ (each having the law as described in Definition~\ref{def:proca_field}) and considering the formal, $\mathfrak{g}$-valued 1-form given by
    \[
    \sum_{j=1}^{d} \bigg(\sum_{\ell = 1}^{n} P_\ell V_\ell \bigg) \, {\rm d} x_j \, ,
    \]
    see Remark~\ref{remark:cordinates_are_independent_for_free_bdy_conditions} below for more details. The random generalized 1-form $\cX_{\mathfrak{g},m}$ will occur as the scaling limit for our lattice Yang-Mills-Higgs theory, as described in Section~\ref{subsec:lattice_YMH_model}. 
    
    \subsection{Main result}
    \label{subsec:main_result}
    
    We are ready to state the main result of the paper, which establishes the $\mathfrak{g}$-valued Proca field as a scaling limit of the lattice Yang-Mills-Higgs model described in Section~\ref{subsec:lattice_YMH_model} under a certain scaling between the gauge coupling and the lattice spacing. 

    Let $\nu = \nu_{\beta,m}$ be an infinite volume limit of the lattice Yang-Mills-Higgs measure given by~\eqref{eq:def_of_YMH_measure_on_torus}. We want to ``lift" a configuration $(U_e) \in G^{E(\bZ^d)}$ sampled from $\nu_{\beta,m}$ into a $\mathfrak{g}$-valued 1-form on $\bR^d$, which we can do via the logarithmic coordinates\footnote{The reason for doing this traces back to the derivation of the lattice gauge theory model. In short, the lattice model is defined with $G$-valued edge variables because it is a discretized version of the continuum $\mathfrak{g}$-valued gauge field, obtained in the following way. If $e = (a, a+e_i)$ for $a \in \bZ^d$ and $e_i$ a standard basis vector, and $A$ is a $\mathfrak{g}$-valued 1-form, then $U_e \approx \exp(cA_i)$, where $c$ is a constant that depends on the lattice spacing and the gauge coupling; for more details, see \cite[Section~3]{chatterjee2016YMSurvey}. To recover a continuum limit, we apply the (matrix) logarithm to move back to $\mathfrak{g}$.} defined in Section~\ref{subsec:logarithmic_coordinates}. Indeed, let $\Log : G\to \mathfrak{g}$ be defined via 
    \begin{equation}
    \label{eq:def_of_Log_with_truncation}
        \Log(U) = \begin{cases}
            \log(U) & U\in {\sf V} , \\ 0 & \text{else}\, ,
        \end{cases}
    \end{equation}
    where we recall that ${\sf V}\subset G$ is a neighborhood of the identity matrix where the exponential map is invertible, and $\log = \exp^{-1}$. The point of~\eqref{eq:def_of_Log_with_truncation} is that we want to define the logarithm globally on $G$ and not worry about branch points or places where the power series in equation~\eqref{eq:log_of_matrix} does not converge. For our particular scaling limit, the value that $\Log$ takes outside of ${\sf V}$ will not play any role, so we might as well set it to be $0$. Indeed, when sampling $(U_e)$ from $\nu$, we will see (Lemma~\ref{lemma:large_values_for_YMH_field_are_rare} below) that when $\beta$ is large, local configurations will consist of matrices in ${\sf V}$, with high probability. 

    Given $(U_e)\in G^E$ sampled from $\nu_{\beta,m}$, we will consider a field $(A_e) \in \mathfrak{g}^E$, given by
    \begin{equation*}
\label{eq:def_of_field_A_lie_algebra}
        A_e = \sqrt{\beta} \, \Log(U_e) \, .
    \end{equation*}
    The random field $A=(A_e)$ will be identified as a random $\mathfrak{g}$-valued 1-form in the following way. For $x\in \bR^d$, let $v(x)\in \bZ^d$ be the closest lattice point to the point $x$ (if there are more than one ``closest" lattice points, we choose the smallest one in lexicographic order). Define
    \begin{equation}
        \label{eq:def_of_Z}
        \qquad Z_j(x) = A_{(v(x),v(x)+e_j)} \, , \qquad 1\le j\le d\, .
    \end{equation}
    The random field $Z:\bR^d \to \mathfrak{g}^d$ acts on test functions $F\in \mathfrak{D}$ in a natural way, namely
    \begin{equation}
        \label{eq:action_of_Z_on_F}
        Z(F) = \sum_{j=1}^{d} \int_{\bR^d} \big\langle Z_j(x),F_j(x)\big\rangle \, {\rm d}x \, .
    \end{equation}
    From this point of view, $Z$ can be considered as a random generalized $\mathfrak{g}$-valued 1-form defined on $\mathfrak{D}$, see the book~\cite[Chapter~III, Section~1.2]{Gelfand-Vilenkin-Vol4}. Similarly, for all $\eps>0$ we set
   \begin{equation*}
        \label{eq:Z_after_scaling_by_eps}
        Z^\eps(x) = \eps^{-(d-2)/2} \,  Z(\eps^{-1}x) \, ,
    \end{equation*}
    and define $Z^\eps(F)$ exactly as in~\eqref{eq:action_of_Z_on_F}. We see that $Z^\eps$ is a random linear functional defined on $\mathfrak{D}$, e.g.\ in the sense of~\cite[Chapter~III]{Gelfand-Vilenkin-Vol4}. The main result of this paper is to identify the $\mathfrak{g}$-valued Proca field as the limit law of $Z^\eps$ in a certain range of asymptotic parameters.
\begin{thm} \label{thm:main_result}
    For $m,\beta,\eps>0$  let $\nu_{\beta,\eps m}$ be an infinite volume limit of~\eqref{eq:def_of_YMH_measure_on_torus} with inverse gauge coupling $\beta>0$ and mass $\eps m$. Assume that $\beta\to \infty$ and $\eps \to 0$ simultaneously such that 
    \begin{equation} \label{eq:assumption_on_rate_of_beta_in_terms_of_eps}
        \beta^{-1} \le \eps^{C_{d,n}} \, ,
    \end{equation}
    for some $C_{d,n}>0$.  
    Then the random generalized $\mathfrak{g}$-valued 1-form $Z^{\eps}$ converges in distribution to $\cX_{\mathfrak{g},m}$, the $\mathfrak{g}$-valued Proca field with mass $m>0$, as described in Definition~\ref{def:g_valued_proca_field}.
\end{thm}
\begin{remark}
    While we do not attempt to pin down an optimal constant in~\eqref{eq:assumption_on_rate_of_beta_in_terms_of_eps}, we remark that it scales linearly in both the lattice dimension $d$ and the $n = \text{dim}(\mathfrak{g})$. For concreteness, taking $C_{d,n} = 100dn$ works.
\end{remark}

\subsection{Complete breakdown of symmetry}
\label{subsec:complete_breakdown_of_symmetry}
    To give some context to Theorem~\ref{thm:main_result}, we explain the relation between the probability measure~\eqref{eq:def_of_YMH_measure_on_torus} and a more general version of the Yang-Mills-Higgs model from particle physics\footnote{also known by the name \emph{The Standard Model}.}. Following~\cite[Chapter~84]{Srednicki-QFT}, this is a Gibbs measure on pairs $(U,\phi)$ which correspond to the action
    \begin{equation*}
        -\beta\sum_{p\in P(\bT_{L}^d)}\mathrm{Re}(\Tr({\sf d}U)_p) -\alpha\sum_{e= (x,y) \in E(\bT_L^d)}\mathrm{Re}(\Tr(\phi_x^*U_e\phi_y)) + \sum_{x \in V(\bT_L^d)}W(|\phi_x|^2) \, .
    \end{equation*}
    Here, $\alpha,\beta\ge 0$ and parameters and $W:[0,\infty] \to \bR$ is some potential. The \emph{Yang-Mills field} is the collection of matrices $U=(U_e)_{e\in E(\bT_{L}^d)}$ taking values in the gauge group $G$ and the \emph{Higgs field} $\phi = (\phi_x)_{x\in V(\bT_{L}^d)}$ takes values in some target manifold $M$. Some common choices in the literature include $M\in\{ \bS^{N-1} ,  G, \mathfrak g\}$, see for instance~\cite{Chatterjee-HiggsMechanism, Fradkin-Shenker, Seiler-book}, and \cite[Section 2.1]{Shen-Zhu-Zhu} for a more exhaustive description of these models. In the present paper, we take $M = G$, and since $G\subset U(N)$, the potential term $W(|\phi_x|^2)$ is constant and can be dropped from the action. Furthermore, working in the regime where $\alpha = \beta m$ is held fixed, we are left with
    \begin{equation} \label{eq:action_before_unitary_gauge}
          -\beta\sum_{p\in P(\bT_{L}^d)}\mathrm{Re}(\Tr({\sf d}U)_p) -\beta m\sum_{e= (x,y) \in E(\bT_L^d)}\mathrm{Re}(\Tr(\phi_x^*U_e\phi_y).
    \end{equation}
    A further reduction is possible once we take into account the gauge symmetry. 
    For any $g: V \to G$, we can associate a \emph{gauge transformation}, which is simply the mapping
        \begin{equation*}
            U_e \mapsto g_x U_eg_y^{-1}, \qquad \phi_x \mapsto g_x \phi_x \,, \qquad e = (x,y) \, ,
        \end{equation*}
    It is not hard to check that both ${\sf d}U_p$ and the probability measure are invariant under theses gauge transformations (see, e.g.,~\cite[Lemma~2.1]{Shen-Zhu-Zhu}). Since only gauge-invariant observables are of physical interest, we may as well consider the model with the gauge fixed. For us, this will be the \emph{unitary gauge}, setting $g_x = \phi_x^{-1}$ for all $x \in V$, which de facto fixes the Higgs field at each vertex to be the identity matrix. After this reductions, it remains to note that for any unitary matrix $U$
    \begin{equation*}
        \|I - U||^2 = \Tr(2I) - 2\mathrm{Re}(\Tr(U)) \, .
    \end{equation*}
   Hence, the Gibbs measure for the action~\eqref{eq:action_before_unitary_gauge}, when fixed to the unitary gauge, has density proportional to $\exp(-\beta \mathcal{H}(U))$, where
    \begin{equation*}
        \mathcal{H}(U) = \frac{1}{2} \sum_{p\in P(\bT_{L}^d)} \| I - ({\sf d}U)_p \|^2 + \frac{m}{2}\sum_{e\in E(\bT_L^d)} \| I - U_e\|^2 \, , 
    \end{equation*}
    which is precisely the model~\eqref{eq:def_of_YMH_measure_on_torus} studied in this paper.

    \subsection{Related works}
    \label{subsec:related_works_intro}
    Our work is most closely related to Chatterjee's recent paper~\cite{Chatterjee-HiggsMechanism}, and we refer the reader to~\cite[Section 3.3]{Chatterjee-HiggsMechanism} for further references to the literature on scaling limits of lattice Yang-Mills-Higgs theories. In~\cite{Chatterjee-HiggsMechanism}, Chatterjee constructs a similar scaling limit as in Theorem~\ref{thm:main_result}, when the gauge group $G$ is either $U(1)$ or $SU(2)$. In this regard, one can view the present work as an extension of~\cite{Chatterjee-HiggsMechanism}. 
    Recall that both $U(1)$ and $SU(2)$ are diffeomorphic to spheres ($\bS^1$ and $\bS^{3}$, respectively), and that these are the only spheres with group structure compatible with their smooth structure (i.e., Lie groups). This subtle property is used in~\cite{Chatterjee-HiggsMechanism} in order to get the Proca field as a scaling limit, as the lifting of the edge configuration is done via the stereographic projection (which differs from our use of logarithmic coordinates). While the stereographic projection is only defined for spheres, the logarithmic coordinates work for a general Lie group $G$. Further, one may argue that using logarithmic coordinates is the ``correct" way to get the scaling limit, as can be seen from the derivation of the lattice Yang-Mills model, in which the edge variables are derived as exponentials of a $\mathfrak{g}$-valued gauge configuration. Even though the stereographic projection and logarithmic map are generally different, they result in the same scaling limit, since their first order linear approximation near the identity element of $G$ is the same.
    
    As we already mentioned, the limiting continuum object for our main result is the \textit{Proca field} (or rather, its $\mathfrak{g}$-valued version), which has been recurrent in the gauge theory literature in both physics and mathematics. The Proca equations were first written down in the physics literature by Romanian physicist Alexandru Proca~\cite{Proca1936theorie}, and they describe a massive spin-1 vector boson. From the mathematical physics point of view, it has been studied in~\cite{Gross1974FreeProca, GinibreVelo1975EuclideanMassiveField, Yao1975ProcaField}, and more recently in~\cite{Chatterjee-HiggsMechanism, ChatterjeeYakir2025Decay}.
    The Proca field is in fact an example of a generalized Gaussian differential form, see the recent survey~\cite{Cao-Sheffield-FGF-overview} for an overview of the theory of Gaussian generalized forms. and their context in Yang-Mills theory. In fact, in~\cite[Lemma 5.6]{Cao-Sheffield-FGF-overview} the (massive) Proca field (as in Definition~\ref{def:proca_field}) is defined directly in the continuum, in contrast to~\cite{Chatterjee-HiggsMechanism}, where the field is constructed via a limit of lattice Proca fields. Both approaches work nicely for the $\mathfrak{g}$-valued Proca field (Definition~\ref{def:g_valued_proca_field}) and have their benefits, and it is the latter approach that we take in this paper.

\subsection*{Notation}
   To ease on the readability, we conclude the introduction with a list of notation that are used throughout the paper. 
    \begin{itemize}
        \item $G\subset U(N)$ a connected, matrix Lie group; $\mathfrak{g}$ its Lie algebra;
        \vspace{0.5mm}
        \item $\exp:\mathfrak{g} \to G$ the exponential map; 
        \vspace{0.5mm}
        \item $\Log:\mathfrak{g}\to G$ the logarithmic map with a truncation, see~\eqref{eq:def_of_Log_with_truncation};
        \vspace{0.5mm}
        \item $\langle A,B \rangle$ Hilbert-Schmidt inner product given by~\eqref{eq:HS_inner_product}; $||A|| = \langle A, A \rangle^{1/2}$ is the induced norm;
%        \vspace{0.5mm}
        \item $\mathcal{A}(\bR^d)$ space of Schwartz 1-forms; $(F,G)$ the inner product on $\mathcal{A}(\bR^d)$ given by~\eqref{eq:inner_product_on_1_forms};
        \vspace{0.5mm}
        \item $R_m : \mathcal{A}(\bR^d) \to \mathcal{A}(\bR^d)$ differential operator given by~\eqref{eq:def_of_R_m};
        \vspace{0.5mm}
        \item $\mathfrak{D}$ space of compactly supported, $C^\infty$-smooth functions from $\bR^d$ to $\mathfrak{g}^d$;
        \vspace{0.5mm}
        \item ${\tt Leb}$ is the Lebesgue measure on $\mathfrak{g}$;
        \vspace{0.5mm}
        \item $\cX_{\mathfrak{g},m}$ is the $\mathfrak{g}$-valued Proca field on $\bR^d$ with mass $m>0$ (Definition~\ref{def:g_valued_proca_field}). 
    \end{itemize}
    We will use the Landau notation $O,o,\Theta$ freely to denote inequalities and limits up to non-asymptotic constants. We will also write $X\lesssim Y$ if $X = O(Y)$. 

\subsection*{Acknowledgments} 
We thank Sourav Chatterjee for very helpful discussions. F.R.\ is supported in part by NSF Graduate Research Fellowship Program DGE-2146755. O.Y.\ is supported in part by NSF grant DMS-2401136. Y.Z.\ is supported in part by NSF grant DMS-2348142, DMS-2450608 and by BSF grant 2024020.

\section{Breakdown of the proof}
\label{sec:breakdown_of_the_proof}
\noindent
In this section, we start with breaking down the proof of Theorem~\ref{thm:main_result} into smaller, more manageable steps. Following~\cite[Chapter~III]{Gelfand-Vilenkin-Vol4}, to prove Theorem~\ref{thm:main_result} we actually need to show that for any fixed $F \in \mathfrak{D}$ we have that 
\begin{equation*}
Z^\eps(F) \xrightarrow{\eps \to 0} \cX_{\mathfrak{g},m}(F)     
\end{equation*}
in law. Broadly speaking, this limit will be obtained in two steps. First, we will show that the lattice $1$-form $Z^\eps$ induced from the Yang-Mills-Higgs measure $\nu_{\beta,\eps m}$ is in fact close (in total variation distance) to a particular random Gaussian 1-form defined on the same lattice (the lattice Proca field, defined below). After that, the desired limit will follow once we show that the latter converge (in law) to the $\mathfrak{g}$-valued Proca field, as the lattice spacing shrinks and the inverse coupling constant grows large. 

\subsection{Gibbs measures on the lattice}
\label{subsec:gibbs_meaures_on_lattice}
Let $L\ge 1$ and denote by $Q_L = [-L,L]^{d}\cap \bZ^d$. We will denote by $E(Q_L)$ and $P(Q_{L})$ the edges and the plaquettes contains in $Q_L$, respectively. Let $\partial Q_L$ denote the boundary edges in $Q_L$ (that is, those edges which connect a vertex from $Q_L$ and $Q_{L}^c$). We now define the lattice Yang-Mills-Higgs measure with prescribed boundary conditions.  
    \begin{definition}
    \label{def:YMH_measure_on_lattice_with_given_boundary_conditions}
        Given $\beta>0$ and a configuration $\boldsymbol{\gamma} = (\gamma_e)_{e\in \partial Q_L}$ with $\gamma_e\in G$, the lattice Yang-Mills-Higgs gauge field (with boundary conditions $\boldsymbol{\gamma}$, mass $m$ and inverse coupling $\beta>0$) is a configuration $(U_e)\subset G^{E(Q_L)}$ sampled from the probability measure
        \[
        {\rm d}\bP_{\beta,m,\boldsymbol{\gamma}}^{{\sf YMH}}(U) = \frac{1}{Z_{\beta,m,\boldsymbol{\gamma}}^{{\sf YMH}}} \,  \exp\big(-\beta\mathcal{H}(U)\big) \prod_{e\in E(Q_L)} {\rm d} \mu(U_e) \prod_{e\in \partial Q_L } {\rm d}\delta_{\gamma_e}(U_e) \, ,
        \]
        where $\mathcal{H}$ is the lattice Yang-Mills-Higgs action given by~\eqref{eq:YMH_action_with_mass_m}. 
    \end{definition}
    \noindent
    The motivation for Definition~\ref{def:YMH_measure_on_lattice_with_given_boundary_conditions} is clear: For $\nu_{\beta,m}$ an infinite volume limit of~\eqref{eq:def_of_YMH_measure_on_torus}, the domain Markov property (see, e.g.~\cite[Chapter~2]{georgii2011gibbs}) implies that the law of $\nu_{\beta,m}$ in $Q_L$ conditioned to have boundary values $\boldsymbol{\gamma}$ is exactly $\bP_{\beta,m,\boldsymbol{\gamma}}^{{\sf YMH}}$. The basic idea towards the proof of Theorem~\ref{thm:main_result} is that on a lattice, when $\beta$ is large, the Yang-Mills-Higgs measure is ``close" to a certain Gaussian field which we now define. In fact, this field is exactly the lattice analogue of the $\mathfrak{g}$-valued Proca field as described in Definition~\ref{def:g_valued_proca_field}. Recall that $\|\cdot\|$ is the Hilbert-Schmidt norm~\eqref{eq:HS_inner_product} and ${\tt Leb}$ is the $n$-dimensional Lebesgue measure on $\mathfrak{g}$. 

     \begin{definition}
        \label{def:lattice_g_valued_proca_field}
        For $L\ge 1$, $m,\beta>0$ and a configuration $\boldsymbol{\eta} = (\eta_e)_{e\in \partial Q_L}$ with $\eta_e\in \mathfrak{g}$, the \emph{lattice $\mathfrak{g}$-valued $d$-dimensional Proca field} (with boundary conditions $\boldsymbol{\eta}$, mass $m$ and inverse coupling $\beta>0$) is a configuration $(X_e)\subset \mathfrak{g}^{E(Q_L)}$ sampled from the probability measure
        \[
        {\rm d} \bP_{\beta,m,\boldsymbol{\eta}}^{\sf \mathfrak{g}}(X) = \frac{1}{Z_{\beta,m,\boldsymbol{\eta}}^{\sf \mathfrak{g}}} \, \exp\big(-\beta S(X)\big)  \prod_{e\in E(Q_L)} {\rm d} \, {\tt Leb}( X_e) \prod_{e\in \partial Q_L} {\rm d}\delta_{\eta_e}(X_e) \, ,
        \]
        where,
        \begin{equation} \label{eq:def_of_gaussian_action}
            S(X) =  \frac{1}{2}\sum_{p\in P(Q_L)} \| ({\sf d}X)_p \|^2 + \frac{m}{2} \sum_{e\in E(Q_L)} \|X_e \|^2 \, ,
        \end{equation}
        and, for a plaquette $p\in P(Q_L)$ consisting of the edges $e_1,e_2,e_3,e_4$, we have
        \begin{equation*}
            ({\sf d}X)_p = X_{e_1} + X_{e_2} - X_{e_3} - X_{e_4} \, .
        \end{equation*}
        We will also consider the $\mathfrak{g}$-valued Proca field with \emph{free} boundary conditions, denoted by $\bP_{\beta,m,\text{\normalfont free}}^{\sf \mathfrak{g}}$.
    \end{definition}
    \begin{remark} \label{remark:cordinates_are_independent_for_free_bdy_conditions}
        For free boundary conditions, the coordinates (i.e., the projections of the Proca field onto each element of an orthonormal basis of $\mathfrak{g}$) are independent and identically distributed. Further, each projection has the distribution of a $d$-dimensional Euclidean lattice Proca field, as defined in \cite{Chatterjee-HiggsMechanism}. Indeed, let $V_1, \dots, V_n$ be an orthonormal basis of $\mathfrak{g}$ with respect to the Hilbert-Schmidt norm, and write $X_e = \sum_{i=1}^n g_{i,e} V_i$. Then 
        \begin{align*}
            S(X) &= \frac{1}{2}\sum_{p\in P(Q_L)} \| \sum_{i=1}^n (g_{i,e_1} + g_{i,e_2} - g_{i,e_3} - g_{i,e_4})V_i\|^2 + \frac{m}{2} \sum_{e\in E(Q_L)} \|\sum_{i=1}^n g_{i,e}V_i\|^2 \\
            &= \sum_{i=1}^n\left( \frac{1}{2}\sum_{p\in P(Q_L)} (g_{i,e_1} + g_{i,e_2} - g_{i,e_3} - g_{i,e_4})^2 + \frac{m}{2} \sum_{e\in E(Q_L)} (g_{i,e})^2\right)
        \end{align*}
        by the orthonormality of $V_1, \dots, V_n$.
    \end{remark}    
    \noindent
    Recall that definition of the truncated logarithm $\Log:G\to \mathfrak{g}$, given by~\eqref{eq:def_of_Log_with_truncation}. Given a configuration $U = (U_e)_{e\in E(Q_L)}$ of elements from the Lie group $G$, we will denote by 
    \begin{equation}
        \label{eq:def_of_map_L}
        {\sf L}(U) = \big(\Log(U_e) \big)_{e\in E(Q_L)} \, , \qquad {\sf L} : G^{E(Q_L)} \to \mathfrak{g}^{E(Q_L)} \, .
    \end{equation}
    Clearly, ${\sf L}$ is measurable, and so we denote by ${\sf L}_\ast \bP_{\beta,m,\boldsymbol{\gamma}}^{{\sf YMH}}$ as the push-forward of the Yang-Mills-Higgs lattice measure (Definition~\ref{def:YMH_measure_on_lattice_with_given_boundary_conditions}) onto a probability measure on $\mathfrak{g}^{E(Q_{L})}$. We further denote by 
    \[
    \cE_2 = \big\{ \forall e\in \partial Q_L  \, : \, \| I - U_e \| \le \beta^{\kappa/2 - 1/2} \big\} \, ,
    \]
    as the event of having ``good" boundary conditions, and further denote by $\cA_2 = {\sf L}(\cE_2)$. We will see later that $\cE_2$ is typical under the Yang-Mills-Higgs measure, but for now let us state the next proposition, which is key in our analysis.
       \begin{proposition}
        \label{prop:YMH_is_close_to_Proca}
        Let $L\ge 1$, $m\in (0,1)$ and $\beta\ge \beta_0(n,d)$ be large enough. Denote by $$\bP_{\beta,m,\cE_2}^{\sf YMH} = \nu_{\beta,m} (\cdot \mid \cE_2)$$ an infinite volume limit of~\eqref{eq:def_of_YMH_measure_on_torus} conditioned on $\cE_2$, and further denote by $$\bP_{\beta,m,\mathcal{A}_2}^{\sf g} = \bP_{\beta,m,\text{\normalfont free}}^{\sf g} (\cdot \mid \mathcal{A}_2)$$ the $\mathfrak{g}$-valued Proca field conditioned on $\cA_2$. Then 
        \[
        d_{{\sf TV}}\Big(\bP_{\beta,m,\mathcal{A}_2}^{\sf g} , \, {\sf L}_\ast \bP_{\beta,m,\cE_2}^{\sf YMH} \Big) \le C L^d \, \frac{\log \beta}{\beta^{2\kappa}} + C L^d \exp\Big(-c m \beta^{2\kappa}\Big) \, ,
        \]
        where $c,C>0$ depend only on the lattice dimension $d\ge 2$ and $n=\text{dim}(G)$.
    \end{proposition}
    
    \noindent
    Here and everywhere, $d_{{\sf TV}}$ denotes the total variation distance between two probability measures on the Euclidean space $\mathfrak{g}^{E(Q_L)}$, equipped with the corresponding Borel sigma-algebra $\mathcal{B}$. Recall that for two probability measures $\nu_1,\nu_2$ on $\mathfrak{g}^{E(Q_L)}$, the total variation distance is defined via
    \begin{equation*}
        d_{{\sf TV}}\big(\nu_1,\nu_2\big) = \sup_{E\in \mathcal{B}} \big|\nu_1(E) - \nu_2(E)\big| \, .
    \end{equation*} 
    The proof of Proposition~\ref{prop:YMH_is_close_to_Proca} is given in Section~\ref{sec:proca_approximation_on_lattice}. 
    
\subsection{Discrete to continuum}
\label{subsec:discrete_to_continuum}
    Another step in the proof of Theorem~\ref{thm:main_result} is to show that the properly normalized lattice $\mathfrak{g}$-valued Proca field (Definition~\ref{def:lattice_g_valued_proca_field}) converges in law to the continuum $\mathfrak{g}$-valued Proca field (Definition~\ref{def:g_valued_proca_field}) as the lattice spacing shrinks. To explain this point more formally, we first need to discuss the random 1-form induced by the lattice $\mathfrak{g}$-valued Proca field. As this is completely analogous to the 1-form induced by the Yang-Mills-Higgs field as described in Section~\ref{subsec:main_result}, we will keep this explanation brief. 

    For $L\ge 1$ let $(X_e)_{e\in E(Q_L)}$ be a realization of the lattice $\mathfrak{g}$-valued Proca field with free boundary conditions $\bP_{\beta,\eps m,\text{free}}^{\sf \mathfrak{g}}$, as described in Definition~\ref{def:lattice_g_valued_proca_field}. Extend it to a random field over $E(\bZ^d)$ by setting $X_e = 0$ for $e\not\in E(Q_L)$. This random field can be identified as a random $\mathfrak{g}$-valued 1-form on $\bR^d$, exactly as in Section~\ref{subsec:main_result}. Indeed, for $x\in \bR^d$, let $v(x)\in \bZ^d$ be the closest lattice point to $x$ and set 
    \begin{equation*}
        \label{eq:def_of_Z_j_gaussian}
        \qquad Z_{j}^{\sf \mathfrak{g}}(x) = \sqrt{\beta} \,  X_{(v(x),v(x)+e_j)} \, , \qquad \text{for} \quad  1\le j\le d \, .
    \end{equation*}
   And so, $Z^{\sf \mathfrak{g}}:\bR^d\to \mathfrak{g}^d$ is a random Gaussian $1$-form, which is analogous to the $1$-form~\eqref{eq:def_of_Z} obtained from lifting the gauge field. For $\eps>0$ the scaling is the same as before
    \begin{equation}
    \label{eq:def_of_Z_G_eps}
        Z^{\sf \mathfrak{g},\eps}(x) = \eps^{-(d-2)/2}  \, Z^{\sf \mathfrak{g}}(\eps^{-1}x)  \, ,
    \end{equation}
    and furthermore, for $F\in \mathfrak{D}$ we set
    \begin{equation*}
        \label{eq:Z_G_eps_acting_on_test_function_F}
        Z^{\sf \mathfrak{g},\eps}(F) = \sum_{j=1}^{d} \int_{\bR^d} \big\langle Z_j^{\sf \mathfrak{g},\eps}(x) , F_{j}(x) \big\rangle \, {\rm d}x \, .
    \end{equation*}
    
    \begin{proposition}
    \label{prop:lattice_Proca_converge_to_cont}
        For all $\delta>0$, let $L = \lfloor \eps^{-1-\delta}\rfloor$, and let $Z^{{\sf \mathfrak{g}},\eps}$ be the random 1-form~\eqref{eq:def_of_Z_G_eps} induced from the Gaussian probability measure $\bP_{\beta,\eps m,\text{\normalfont free}}^{\sf \mathfrak{g}}$. Assuming that~\eqref{eq:assumption_on_rate_of_beta_in_terms_of_eps} holds, we have
        \[
        \qquad Z^{{\sf \mathfrak{g}},\eps}(F) \xrightarrow{\quad } \cX_{\mathfrak{g},m}(F) \qquad \text{in law,}
        \]
         for all $F\in \mathfrak{D}$ as $\eps\to 0$ and $\beta\to\infty$ simultaneously. Here, $\cX_{\mathfrak{g},m}$ is the $\mathfrak{g}$-valued Proca field on $\bR^d$, as described in Definition~\ref{def:g_valued_proca_field}. 
    \end{proposition}
    \noindent
    The proof of Proposition~\ref{prop:lattice_Proca_converge_to_cont} is proved in Section~\ref{sec:cont_proca_as_lattice_spacing_shrinks} below. We mention that Proposition~\ref{prop:lattice_Proca_converge_to_cont} is the $\mathfrak{g}$-valued version of an analogous statement in~\cite[Theorem~4.6]{Chatterjee-HiggsMechanism} regarding the Euclidean Proca field. In turn, this extension is somewhat routine, and our analysis will be heavily based on the derivation from~\cite{Chatterjee-HiggsMechanism}.

    \noindent
    Assuming Proposition~\ref{prop:YMH_is_close_to_Proca} and Proposition~\ref{prop:lattice_Proca_converge_to_cont}, we are already pretty close to proving our main result Theorem~\ref{thm:main_result}. There are two required steps to complete the proof:
    \begin{enumerate}
        \item[1.] Showing that the event $\cE_2$ for having ``good" boundary conditions (see~\eqref{eq:def_of_E}) of the Yang-Mills-Higgs measure occurs with high probability for $\beta$ large (see Lemma~\ref{lemma:large_values_for_YMH_field_are_rare} below, which proves a slightly stronger statement); and
        \item[2.]  Showing that for the lattice Proca field, a small perturbation of the boundary conditions has a negligible effect on the distribution of the random 1-form $Z^{{\sf G},\eps}$.
    \end{enumerate}
    While these and other small steps will be handled in what follows, we remark that in Section~\ref{sec:cont_proca_as_lattice_spacing_shrinks} we shall combine all of the ingredients, and finally provide the proof of Theorem~\ref{thm:main_result}.
    
\section{Proca approximation on the lattice}
\label{sec:proca_approximation_on_lattice}

\noindent 

The goal of this section is to prove Proposition~\ref{prop:YMH_is_close_to_Proca}.
As the inverse coupling constant $\beta>0$ and the mass $m>0$ will remain fixed in this section, we will lighten on the notation and not indicate the dependence on these parameters for different quantities.

Let $\nu=\nu_{\beta,m}$ be an infinite volume limit of the periodic Yang-Mills-Higgs lattice measure~\eqref{eq:def_of_YMH_measure_on_torus}. Recalling that $Q_L$ is a cube of side-length $L\ge 1$ in the lattice $\bZ^d$, we denote by $\bP_{\boldsymbol{\gamma}}^{\sf YMH} = \bP_{\beta,m,\boldsymbol{\gamma}}^{\sf YMH} $ the lattice Yang-Mills-Higgs measure with boundary conditions $\boldsymbol{\gamma}$ (Definition~\ref{def:YMH_measure_on_lattice_with_given_boundary_conditions}). 

For $\kappa = \kappa(d,n)\in(0,\tfrac{1}{2})$ small but fixed, we will consider the events 
\begin{align}    \label{eq:def_of_E}
    \nonumber \cE_1 &= \big\{ \forall e\in E(Q_L)   \, : \, \| I - U_e \| \le \beta^{\kappa-1/2} \big\} \, , \\ \cE_2 &= \big\{ \forall e\in \partial Q_L  \, : \, \| I - U_e \| \le \beta^{\kappa/2 - 1/2} \big\} \, ,  
\end{align}
and $\cE = \cE_1 \cap \cE_2$.
Note that, for $\beta$ large enough (depending on $G$), we have the obvious inclusion
\[
\big\{ \| I - U \| \le \beta^{\kappa-1/2}\big\} \subset {\sf V} \, ,
\]
where we recall that ${\sf V}\subset G$ is the neighborhood of the identity where the logarithmic map~\eqref{eq:def_of_Log_with_truncation} is a local diffeomorphism. We further set $\cA = {\sf L}(\cE)$ and $ \cA_2 = {\sf L}(\cE_2)$, where ${\sf L}$ is given by~\eqref{eq:def_of_map_L}, and observe that
\begin{equation*}
    {\sf L}_\ast \bP_{\boldsymbol{\gamma}}^{\sf YMH} (\cA) = \bP_{\boldsymbol{\gamma}}^{\sf YMH} (\cE) \, . 
\end{equation*}
Here, we have in mind the boundary conditions $\boldsymbol{\gamma}$ prescribed in Proposition~\ref{prop:YMH_is_close_to_Proca}, though the above equality holds for any boundary conditions. As we will encounter several probability measures throughout this section, both on $G^{E(Q_L)}$ and $\mathfrak{g}^{E(Q_L)}$, we provide in Table~\ref{table:table} a list of those for the reader's convenience:

\begin{table}[H]
\setlength{\tabcolsep}{12pt}
\renewcommand{\arraystretch}{1.3}
\centering
\begin{tabular}{lll}
\toprule
\# & \textbf{Measure} & \textbf{Description} \\
\midrule
1 & $\bP_{\cE_2}^{\sf YMH} = \nu_{\beta,m}(\cdot \mid \cE_2) $ & YMH measure conditioned on the events $\cE_2$ (given by \eqref{eq:def_of_E}) \\
2 & $\bP_{\cE}^{\sf YMH} = \bP_{\cE_2}^{\sf YMH}(\cdot \mid \cE_1 ) $ & YMH measure conditioned on the events $\cE = \cE_1 \cap \cE_2$ \\
3 & ${\sf L}_\ast\bP_{\mathcal{E}_2}^{\sf YMH}$ & Push-forward of $\bP_{\mathcal{E}_2}^{\sf YMH}$ under the map ${\sf L}$ \\
4 & ${\sf L}_\ast\bP_{\mathcal{E}}^{\sf YMH}$ & Push-forward of $\bP_{\mathcal{E}}^{\sf YMH}$ under the map ${\sf L}$ \\
%5 & ${\sf L}_\ast\bP_{\boldsymbol{\gamma},\mathcal{E}}^{\sf YMH}$ & Push-forward of $\bP_{\boldsymbol{\gamma},\mathcal{E}}^{\sf YMH}$ under the map ${\sf L}$ \\
5 & $\bP_{\cA}^{{\mathfrak{g}}} = \bP_{\beta,m,\text{free}}^{{\mathfrak{g}}}(\cdot \mid\cA )$ & Lattice Proca field (Definition~\ref{def:lattice_g_valued_proca_field}) conditioned on the event $\cA$\\
6 & $\bP_{\cA_2}^{{\mathfrak{g}}} = \bP_{\beta,m,\text{free}}^{{\mathfrak{g}}}(\cdot \mid\cA_2 )$ & Lattice Proca field conditioned on the event $\cA_2$ \\
\bottomrule
\end{tabular}

\caption{List of probability measures appearing in Section~\ref{sec:proca_approximation_on_lattice}.\label{table:table} }
\end{table}
\noindent 
Proposition~\ref{prop:YMH_is_close_to_Proca} asserts that the total variation distance between (3) and (6) is small. In turn, this will be a consequence of the fact that
\[
d_{\sf TV}\big((3),(4)\big) \, , \qquad d_{\sf TV}\big((4),(5)\big) \, , \qquad d_{\sf TV}\big((5),(6)\big) \, , 
\]
are all small, and the $d_{\sf TV}$ satisfies a triangle inequality. 
\subsection{Some preparations}
As a first step, we state a simple lemma, which in particular shows that the condition on the boundary values~\eqref{eq:def_of_E} in the statement of Proposition~\ref{prop:YMH_is_close_to_Proca} is typical.
\begin{lemma}\label{lemma:large_values_for_YMH_field_are_rare}
    Let $\nu$ be an infinite volume Yang-Mills-Higgs measure with inverse coupling $\beta\ge 1$ and mass $m\in(0,1)$, and let $L\ge 1$. Then, for all $\kappa>0$, we have
    \[
    \nu\Big(\exists e\in E(Q_L) \cup \partial Q_L \, : \, \|I - U_e \| \ge \beta^{\kappa-1/2}\Big) \le C L^d \frac{\log \beta}{\beta^{2\kappa}} \, .
    \]
\end{lemma}
\noindent
Before proving the lemma, we state another simple claim that will be useful throughout.
\begin{claim} \label{claim:taylor_expansion_for_log}
    For all $U\in G$ such that $\| I - U \|\le 1/2$, we have
    \[
    \frac{1}{2} \|I - U \| \le \| \log U\| \le 2 \, \|I - U \| \, .
    \]
\end{claim}
\begin{proof}
Since
    \begin{equation*}
		\log (U) = -\sum_{k\ge 1} \frac{(I-U)^k}{k} \, ,
	\end{equation*} 
both inequalities follows from the inequality
\[
\sum_{k=2}^\infty \frac{\| I -U \|^k}{k} \le \|I-U \| \, \sum_{k\ge 1} 2^{-k} \le \frac{\| I-U \| }{2}\, .
\]
\end{proof}
\begin{proof}[Proof of Lemma~\ref{lemma:large_values_for_YMH_field_are_rare}]
    Since $\nu$ is an infinite volume limit with periodic boundary conditions, it is automatically translation invariant. Therefore, by union bounding over the edges from $E(Q_L)\cup \partial Q_L$ and applying Markov's inequality, the lemma reduces to
        \begin{equation}
            \label{eq:pointwise_L_2_estimate_under_nu}
            \bE_\nu  \| I - U_{e^\prime} \|^2 \le C \frac{\log \beta}{\beta}  \, ,
        \end{equation}
        where $e^\prime$ is a symbolic edge from $E(\bZ^d)$. In turn,~\eqref{eq:pointwise_L_2_estimate_under_nu} would follow once we show that
        \begin{equation}
        \label{eq:pointwise_L_2_estimate_under_nu_on_finite_torus}
            \sup_{L\ge 1} \, \bE_{\nu_L}  \| I - V_{e^\prime} \|^2 \le C \frac{\log \beta}{\beta}  \, ,
        \end{equation}
        where we recall that $\nu_L$ is the Yang-Mills-Higgs measure on the torus $\bT_L^d$, given by~\eqref{eq:def_of_YMH_measure_on_torus}. The proof of~\eqref{eq:pointwise_L_2_estimate_under_nu_on_finite_torus} is based on the stability of the Yang-Mills-Higgs action $\mathcal{H}$. Recalling that $\mu$ is the (normalized) Haar measure for the gauge group $G$ and that $n=\text{dim}(G)$, Claim~\ref{claim:taylor_expansion_for_log} yields that 
        \[
        \beta^{-n/2}\lesssim \mu\big( V\in G \, : \, \|I-V \|^{2} \le \beta^{-1} \big) \lesssim \beta^{-n/2} \, ,
        \]
        for all $\beta\ge 1$. Plugging this into the partition function in~\eqref{eq:def_of_YMH_measure_on_torus} yields 
        \begin{align*}
            Z_L^{\sf YMH} &= \int_{G^{E(\bT_L^d)}} e^{-\frac\beta2 \sum_{p\in P(\bT_{L}^d)} \| I - ({\rm d}U)_p \|^2  - \frac{\beta m}{2} \sum_{e\in E(\bT_{L}^d)} \| I - U_e\|^2}   \prod_{e\in E(\bT_L^d)} {\rm d}\mu(U_e) \\ & \ge \int_{G^{E(\bT_L^d)}}  \, e^{-\frac\beta2 \sum_{p\in P(\bT_L^d)} \| I - ({\rm d}U)_p \|^2  - \frac{\beta m}{2} \sum_{e\in E(\bT_L^d)} \| I - U_e\|^2}   \prod_{e\in E(\bT_L^d)} \mathbf{1}_{\{\|I-U_e\|^2 \le \beta^{-1} \}} \, {\rm d}\mu(U_e) \\ & \gtrsim e^{- 8 |P(\bT_{L}^d)|  - \frac{m}{2} |E(\bT_L^d)|} 
            \cdot \big(\beta^{-n/2} \big)^{|E(\bT_L^d)|} \, .
        \end{align*}
        As we assume that $m\in(0,1)$, we get that
        \[
        \bE_{\nu_L}\big[ e^{\beta \mathcal{H}(U)}\big] = \frac{1}{Z_L^{\sf YMH}} \le \big(C\beta^{n/2}\big)^{L^d} \, ,
        \]
        for some $C>0$. With this bound on the exponential moment, Markov's inequality gives that 
        \[
        \nu_L\big(\mathcal{H}(U) \ge t/\beta \big) \le e^{-t} \big(C\beta^{n/2}\big)^{L^d} 
        \]
        for all $t\ge 0$, which in turn show that 
        \begin{align*}
            \beta \, \bE_{\nu_L} \big[\mathcal{H}(U) \big] &= \int_{0}^{\infty} \nu_L\big(\mathcal{H}(U) \ge t/\beta \big) \, {\rm d}t \\ & \le \int_{0}^{\infty} \min\Big\{1, e^{-t} \big(C\beta^{n/2}\big)^{L^d} \Big\}\, {\rm d}t \lesssim L^d \log\big(C\beta^{n/2} \big) \, .
        \end{align*}
        On the other hand, we also have $$ \mathcal{H}(U) \ge \frac{1}{2} \sum_{e\in E(\bT_L^d)} \| I - U_e\|^2 \, , $$ and~\eqref{eq:pointwise_L_2_estimate_under_nu_on_finite_torus} follows from symmetry of the torus. This proves~\eqref{eq:pointwise_L_2_estimate_under_nu} and with that the lemma. 
\end{proof}
\noindent
As a first consequence of Lemma~\ref{lemma:large_values_for_YMH_field_are_rare}, we get a bound on $d_{\sf TV}((1),(2))$, where (1) and (2) are the measures given in Table~\ref{table:table}. This bound immediately translates to a bound on the distance between (3) and (4).

\begin{claim} \label{claim:from_3_to_4}
    Assume that 
    \begin{equation*}
    C L^d \, \frac{\log \beta}{\beta^{\kappa}} \le \frac{1}{10} \, ,     
    \end{equation*}
    where $C$ is the constant on the right-hand side of Lemma~\ref{lemma:large_values_for_YMH_field_are_rare}. Then 
    \[
    d_{\sf TV}\big({\sf L}_\ast \bP_{\cE_2}^{\sf YMH} , {\sf L}_\ast\bP_{\mathcal{E}}^{\sf YMH} \big) \le  2 C L^d \, \frac{\log \beta}{\beta^{2\kappa}} \, .
    \]
\end{claim}
\begin{proof}
    Let $\cP$ be any probability measure, $E$ some event with $\cP(E)>0$ and $\cP_E$ be the probability measure conditioned on the event $E$,  then
    \begin{equation}
    \label{eq:bound_on_total_variation_distance_of_conditional_vs_original}
        d_{\sf TV}(\cP,\cP_E) = \cP(E^{c}) \, .
    \end{equation}
    Indeed, for any other event $A$ we have 
    \[
    \cP(A) - \cP_E(A) = \cP(A\cap E^c) + \cP(A\cap E) - \frac{\cP(A\cap E)}{\cP(E)} = \cP(A\cap E^c) - \cP(A\cap E)  \, \frac{\cP(E^c)}{\cP(E)} \, .  
    \]
    Now, it is not hard to check that
    \[
    \max_{\substack{x\in [0, \cP(E)] \\ y\in [0,\cP(E^c)]}} \Big| y- x \,  \frac{\cP(E^c)}{\cP(E)} \Big| \le \cP(E^c) \, ,
    \]
    which gives $d_{\sf TV}(\cP,\cP_E) \le \cP(E^{c})$. The equality~\eqref{eq:bound_on_total_variation_distance_of_conditional_vs_original} now follows from plugging $A= E^c$. Lemma~\ref{lemma:large_values_for_YMH_field_are_rare} implies that 
    \[
    \nu(\cE_2^c) \le C L^{d} \frac{\log \beta}{\beta^\kappa} \le \frac{1}{10} \, ,
    \]
    and 
    \[
    \nu\big(  \cE_1^c \cap \cE_2 \big) \le \nu(\cE_1^c) \le CL^d\frac{\log \beta}{\beta^{2\kappa}} \, .
    \]
    Hence,~\eqref{eq:bound_on_total_variation_distance_of_conditional_vs_original} gives that 
    \[
    d_{\sf TV}\big({\sf L}_\ast \bP_{\cE_2}^{\sf YMH} , {\sf L}_\ast\bP_{\mathcal{E}}^{\sf YMH} \big)  \le d_{\sf TV}\big(\bP_{\cE_2}^{\sf YMH} , \bP_{\mathcal{E}}^{\sf YMH} \big) = \bP_{\cE_2}^{\sf YMH}(\cE_1^c) \le \frac{10}{9} C L^d \frac{\log \beta}{\beta^{2\kappa}} \, .
    \] 
\end{proof}

\subsection{An application of the Baker-Campbell-Hausdorff formula}   
The lattice Yang-Mills-Higgs measure is defined on $G^E$, but we need to analyze the pushforward of this measure under the map ${\sf L}$ from (\ref{eq:def_of_map_L}). The next few lemmas enable us to do this analysis. For a matrix $X\in \mathfrak{g}$, we define ${\sf ad}_X:\mathfrak{g} \to \mathfrak{g}$ as the linear map (usually referred to as the Poisson brackets)
    \begin{equation*}
        {\sf ad}_X(Y) = [X,Y] = XY - YX \, .
    \end{equation*} 
    The Baker-Campbell-Hausdorff formula gives a precise expression for $\log(e^X e^Y)$ when $||X||$ and $||Y||$ are small enough. The first few terms are given by 

    $$
    \log(e^X e^Y) = X + Y + \frac{1}{2}[X, Y] + \frac{1}{12}[X, [X, Y]] - \frac{1}{12}[Y, [X, Y]] + \cdots, 
    $$
    and the full formula along with other ways of representing it can be found in~\cite[Chapter~5]{Hall-LieBook}. When $X$ and $Y$ commute, all commutators in the formula vanish, and we are left with $$\log(e^X e^Y) = X + Y \, .$$ Since we are taking $\beta \to \infty$ at the prescribed rate in Theorem \ref{thm:main_result}, this causes the value of $||\log(U_e)||$ to be small for all $e \in E(Q_L)$. Then the higher order terms in the BCH formula (i.e. the commutators) become increasingly small in norm in the scaling limit.  Thus, this scaling limit regime effectively `abelianizes' the problem. The following proposition, which is a simple application of the BCH formula, shows how to differentiate the exponential map. 
    \begin{proposition}[{\cite[Theorem~5.4]{Hall-LieBook}}]
        \label{prop:derivative_of_exp_map}
        Let $X,Y$ be $N\times N$ matrices with complex entries. Then
        \begin{align*}
            \frac{{\rm d}}{{\rm d}t} \exp(X + tY) \bigg|_{t=0} &= e^X \bigg( \frac{I-e^{-{\sf ad}_X}}{{\sf ad}_X}(Y) \bigg) \\ &= e^X\bigg( Y - \frac{[X,Y]}{2!} + \frac{[X,[X,Y]]}{3!} - \ldots  \bigg)
        \end{align*}
    \end{proposition}
    \noindent
    Proposition~\ref{prop:derivative_of_exp_map} allows us to derive the necessary change of variables formula for functions defined in a neighborhood of $I\in G$. 
    \begin{lemma}
       \label{lemma:change_of_variables_lie_group_exponential_coordinates}
        Let $\psi:G\to \bR$ be a measurable function with $\text{\normalfont supp}(\psi) \subset {\sf V}$. Then 
        \[
        \int_{G} \psi(U) \, {\rm d}\mu (U) = \int_{\mathfrak{g}} \psi\big(\exp(X)\big) \, 
        \bigg| \det\Big(\frac{I - e^{-{\sf ad}_X}}{{\sf ad}_X}\Big)\bigg| \, {\rm d} {\tt Leb}(X) \, , 
        \]
        where ${\rm d} {\tt Leb}$ is the $n$-dimensional Lebesgue measure on $\mathfrak{g}$.
    \end{lemma}
    \begin{proof}
        The pair $({\sf V}, \exp)$ defines a local chart on the manifold $G$. By Proposition~\ref{prop:derivative_of_exp_map}, we have that
        \[
        D_X \exp : T_{\exp(X)} G \longrightarrow T_{\exp(X)} G \simeq \mathfrak{g} \, ,
        \]
        is given by
        \[
        D_X \exp (Y) = e^X \Big( \frac{I-e^{-{\sf ad}_X}}{{\sf ad}_X} \Big)(Y) \, .
        \]
        Since $G \subset U(N)$, we know that
        \[
        |\det(e^X)| = 1 \, ,
        \]
        for all $X\in \mathfrak{g}$. Thus, the Haar measure $\mu$ on $G$ (which is the unique volume that is left-invariant) pulled back onto $\mathfrak{g}$ has the Jacobian
        \[
        \Big|\det\Big(\frac{I - e^{-{\sf ad}_X}}{{\sf ad}_X}\Big)\Big| \, ,
        \]
        with respect to the Lebesgue measure on $\mathfrak{g}$;  the corresponding unique left-invariant measure. 
    \end{proof}
    \noindent
    In our proofs we will be integrating over a product Lie group (one copy of $G$ for each edge of the lattice), we have the corresponding result for product Lie groups. The main difference here is that the ${\sf ad}$ operator factors over the product, allowing us to write the determinant of the Jacobian as a product over each factor of the individual determinants.

      \begin{lemma}
    \label{lemma:change_of_variables_product_lie_group}
    Let $E$ be a finite set. Let  $f: G^E \to \bR$ be any bounded measurable function, such that $supp(f) \subset {\sf V}^E$, where ${\sf V}$ is is the domain on which the logarithmic coordinates are defined. Then
    \[
        \int_{G^E} f(\mathbf{U}) \, \prod_{e \in E} {\rm d} \mu_e  (U_e) = \int_{\mathfrak{g}^E} f\big(\exp(\mathbf{X})\big)\, 
         \prod_{e \in E}\Big|\det\Big(\frac{I - e^{-{\sf ad}_{X_e}}}{{\sf ad}_{X_e}}\Big)\Big| \, \prod_{e \in E}{\rm d} {\sf Leb}(X_e).
        \]
    Here, $\mathbf{U} = (U_e)_{e\in E}$ and $\exp(\mathbf{X}) = (\exp(X_e))_{e \in E}$.
    \end{lemma}

    \begin{proof}
        Below we list some standard facts about compact Lie groups and their corresponding Lie algebras. All these facts are simple to prove; we refer the reader to the book~\cite{Bourbaki}. The facts we will use are as follows:
        \begin{enumerate}
            \item If $G$ is a Lie group with Lie algebra $\mathfrak{g}$, then the direct product $G^E$ is a Lie group with Lie algebra $\mathfrak{g}^E =  \bigoplus_{e \in E}\mathfrak{g}_e$; see \cite[Chapter 3, Section 3.8]{Bourbaki}.
            \vspace{1mm}
            
            \item If $\mu$ is the Haar probability measure on $G$, then $\mu^{\otimes E}$ is the Haar probability measure on $G^E$. This follows from uniqueness of the Haar measure and a direct inspection.

            \vspace{1mm}
            
            \item The ${\sf ad}$ operator on $\mathfrak{g}^E$ acts component-wise. In other words, for $\mathbf{X}, \mathbf{Y} \in \mathfrak{g}^E$ we have $${\sf ad}_{\mathbf{X}}(\mathbf{Y}) = (X_eY_e - Y_eX_e)_{e \in E} \, .$$ See~\cite[Chapter 1, Section 1.8]{Bourbaki} for the proof.
        \end{enumerate}
        We conclude that ${\sf ad}_{\mathbf{X}}$ is a block diagonal operator on $\mathfrak{g}^E$, and hence so is $$\frac{I - e^{-{\sf ad}_{\mathbf{X}}}}{{\sf ad}_{\mathbf{X}}} \, .$$ The corresponding determinant factors over the diagonal blocks, and we have
        \[
        \det \left( \frac{I - e^{-{\sf ad}_{\mathbf{X}}}}{{\sf ad}_{\mathbf{X}}}\right) = \prod_{e \in E}\det\Big(\frac{I - e^{-{\sf ad}_{X_e}}}{{\sf ad}_{X_e}}\Big).
        \]
        Applying Lemma \ref{lemma:change_of_variables_lie_group_exponential_coordinates} with the function $f(\mathbf{U})$ and Fubini proves the lemma.
    \end{proof}

\subsection{Density comparison on $\mathfrak{g}$}
To compare the total variation distance between (5) and (6), we shall appeal to a general bound for the total variation distance between Gibbs measures, given by the next claim. 
\begin{claim} \label{claim:total_variation_bound_between_gibbs_measures}
    For $m\ge 1$, suppose that $\varphi_1,\varphi_2: \bR^m \to \bR_{\ge 0}$ are two integrable functions. Let $\mu_1$ and $\mu_2$ be the probability measures on $\bR^m$ obtained by normalizing $\varphi_1$ and $\varphi_2$ by their integrals, respectively. Then
    \[
    d_{\sf TV}(\mu_1,\mu_2) \le \frac{\| \varphi_1 - \varphi_2 \|_1}{\max\big\{ \|\varphi_1\|_1, \,  \|\varphi_2\|_1 \big\}} \, ,
    \]
    where $\|\cdot \|_1$ denotes the $L^1$ norm on $\bR^m$ with respect to the Lebesgue measure. 
\end{claim}
\begin{proof}
    For the proof, denote by ${\sf m}$ the Lebesgue measure on $\bR^m$. From the standard $L^1$ characterization of the total variation distance for absolutely continuous probability measures, we have 
    \begin{equation} \label{eq:total_variation_distance_l_1_formula}
        d_{\sf TV}(\mu_1,\mu_2) = \frac{1}{2} \, \bigg\| \frac{\varphi_1}{\int_{\bR^m} \varphi_1 \, {\rm d}{\sf m}} - \frac{\varphi_2}{\int_{\bR^m} \varphi_2 \, {\rm d}{\sf m}} \bigg\|_1 \, .
    \end{equation}
    By the triangle inequality, we have
    \begin{align*}
        \int_{\bR^m} \bigg| \frac{\varphi_1}{\| \varphi_1 \|_1} - \frac{\varphi_2}{\| \varphi_2 \|_1}\bigg| \, {\rm d}{\sf m} & \le \int_{\bR^m}  \frac{|\varphi_1 - \varphi_2|}{\| \varphi_1 \|_1}  \, {\rm d}{\sf m} + \int_{\bR^m} |\varphi_2| \Big( \frac{1}{\| \varphi_1 \|_1} - \frac{1}{\| \varphi_2 \|_1} \Big) \, {\rm d}{\sf m} \\ & = \frac{\| \varphi_1 - \varphi_2 \|_1}{\| \varphi_1 \|_1} + \frac{\| \varphi_2 \|_1 - \| \varphi_1 \|_1}{\| \varphi_1 \|_1 \cdot \| \varphi_2 \|_1} \cdot \| \varphi_2 \|_1 \\ & \le 2\frac{\| \varphi_1 - \varphi_2 \|_1}{\| \varphi_1 \|_1} \, .
    \end{align*}
    In view of~\eqref{eq:total_variation_distance_l_1_formula}, we get that
    \[
    d_{\sf TV}(\mu_1,\mu_2) \le \frac{\| \varphi_1 - \varphi_2 \|_1}{ \|\varphi_1\|_1} \, , 
    \]
    and a symmetric argument yields the claim.
\end{proof}
\noindent
We also include a few simple claims, which will help us compare the pushforward of the Yang-Mills-Higgs probability density function with the Proca field density. All these claims follow from simple matrix-valued calculus.
\begin{claim}\label{claim:jacobian_of_change_of_variables_close_to_one}
    For all $X\in \mathfrak{g}$ with $\| X\|\le 1$ we have
    \[
    \Big| \, \Big| \det\Big(\frac{I - e^{-{\sf ad}_{X_e}}}{{\sf ad}_{X_e}}\Big)\Big| - 1 \Big| \le C \| X \| \, .
    \]
\end{claim}
\begin{proof}
    Clearly $\det(I) = 1$. Furthermore, we know by the Jacobi formula that for any $N\times N$ matrix $H$ we have
        \[
        D_I \det (H) = \text{Tr}(H) \, .
        \]
        Therefore, Taylor's theorem implies that
        \[
        \Big| \, \Big|\det\Big(\frac{I - e^{-{\sf ad}_X}}{{\sf ad}_{X}}\Big)\Big| - 1  \Big| \le \Big| \text{Tr} \Big(\frac{I - e^{-{\sf ad}_X}}{{\sf ad}_{X}} - I \Big) \Big| \, .
        \]
        Letting $\| \cdot \|_{\text{op}}$ denote the operator norm of a matrix, we always have $\|{\sf ad}_X \|_{\text{op}} \le 2 \| X \|$. In particular
        \[
        \Big\| \frac{I - e^{-{\sf ad}_X}}{{\sf ad}_{X}} - I\Big\|_{\text{op}} \lesssim \| X\| \, ,
        \]
        and conclude by noting that $|\text{Tr}(H)|\le n \| H\|_{\text{op}}$ for any matrix $H$. 
\end{proof}
\begin{claim}\label{claim:taylor_expansion_for_exponential_of_matrix}
For all $X\in \mathfrak{g}$ with $\| X\|\le 1$ we have
\[
\| \exp(X) - I - X \| \le C\| X \|^2 \, .
\]
\end{claim}
\begin{proof}
The Hilbert-Schmidt norm is sub-multiplicative, so we can use~\eqref{eq:exp_of_matrix} and bound as
\[
\Big\| \sum_{k=2}^\infty \frac{X^k}{k!} \Big\| \le \sum_{k=2}^\infty \frac{\| X\|^k}{k!} \lesssim \|X \|^2 \, .
\]
\end{proof}
\begin{claim} \label{claim:identity_minus_product_of_exp_is_quadratic}
    For $X_1,\ldots,X_4\in \mathfrak{g}$ we denote by $ \displaystyle M = \max_{1\le i \le 4} \| X_i\|$. Assuming that $M < 1$, we have
    \vspace{-3mm}
    \[
    \big\| I - \prod_{i=1}^4 \exp(X_i) \big\|^2 = \big\| X_1 + X_2 + X_3 + X_4 \big\|^2 + O\big(M^3\big) \, .
    \]
\end{claim}
\begin{proof}
    We use the power series form of the matrix exponential, and note that all series are absolutely convergent. We have
    \begin{align*}
        \big\| I - \prod_{i=1}^4 \exp(X_i) \big\|^2 &= \bigg\| I - \prod_{i=1}^4 \sum_{j=0}^{\infty}\frac{X_i^j}{j!} \bigg\|^2 \\
        &= \bigg\| I - \sum_{j=0}^{\infty} \sum_{j_1 + j_2 + j_3+ j_4 = j}\frac{X_1^{j_1}X_2^{j_2}X_3^{j_3}X_4^{j_4}}{j_1! j_2! j_3! j_4!} \bigg\|^2 \\
        &= \bigg\|  X_1 + X_2 + X_3 + X_4 - \sum_{j=2}^{\infty} \sum_{j_1 + j_2 + j_3+ j_4 = j}\frac{X_1^{j_1}X_2^{j_2}X_3^{j_3}X_4^{j_4}}{j_1! j_2! j_3! j_4!} \bigg\|^2 \\
        &= \big\|X_1 + X_2 + X_3 + X_4\big\|^2 + O(M^3),
    \end{align*}
    where in the last line we used the fact that the Hilbert-Schmidt norm is sub-multiplicative and the fact that $M < 1$.
\end{proof}
\noindent
We are finally ready to state and prove our density comparison lemma. 
\begin{lemma} \label{lemma:from_4_to_5}
    Let $\boldsymbol{\gamma} = (\gamma_e)_{e\in \partial Q_L}$ with $\gamma_e\in G$ be a configuration such that 
    \begin{equation}
        \label{eq:good_boundary_conditions_in_comparison_lemma}
        \| I - \gamma_e \| \le \beta^{\kappa/2 - 1/2} \, .
    \end{equation}    
    Let $\boldsymbol{\gamma}^\prime = (\gamma_e^\prime)$ be another such configuration, and denote by $\boldsymbol{\eta}^\prime = (\eta_e^\prime)$ where $\eta_e^\prime = \log(\gamma_e^\prime) \in \mathfrak{g}$. Then
    \[
    \sup_{\boldsymbol\gamma , \boldsymbol{\gamma}^\prime} \, d_{\sf TV}\big( {\sf L}_\ast\big(\bP_{\boldsymbol{\gamma}}^{\sf YMH} \cdot \mathbf{1}_{\mathcal{E}}\big) \, , \,  \bP_{\boldsymbol{\eta}^\prime}^{{\mathfrak{g}}} \cdot \mathbf{1}_{\mathcal{A}} \big) \le C L^{d} \beta^{\kappa-1/2}
    \]
    where the supremum is taken over all possible configurations satisfying~\eqref{eq:good_boundary_conditions_in_comparison_lemma}. 

\end{lemma}
\begin{proof}
    By Definition~\ref{def:YMH_measure_on_lattice_with_given_boundary_conditions}, the density of $\bP_{\boldsymbol{\gamma},\mathcal{E}}^{\sf YMH} \cdot \mathbf{1}_{\cE}$ is proportional to 
    \begin{equation*}
        {\rm d }\bP_{\boldsymbol{\gamma}}^{\sf YMH}(U) \cdot \mathbf{1}_{\cE}(U) \propto \exp\big(-\beta\mathcal{H}(U)\big) \cdot \mathbf{1}_{\{ U\in \cE \}} \prod_{e\in E(Q_L)} {\rm d} \mu(U_e) \prod_{e\in \partial Q_L } {\rm d}\delta_{\gamma_e}(U_e) \, ,  
    \end{equation*}
    where $U = (U_e)_{e\in E(Q_L)}$ and $\mathcal{H}$ is the Yang-Mills-Higgs action~\eqref{eq:YMH_action_with_mass_m}. Hence, denoting by $X = {\sf L}(U) $ the lifted configuration on $\mathfrak{g}^{E(Q_L)}$ and by $\eta_e = \log(\gamma_e)$, Lemma~\ref{lemma:change_of_variables_product_lie_group} shows that 
    \vspace{1mm}
    \begin{align} \label{eq:density_of_4}
        &{\rm d } \, {\sf L}_\ast \bP_{\boldsymbol{\gamma}}^{\sf YMH}(X) \cdot \mathbf{1}_{\cA}(X) \propto \\ \nonumber & \exp\Big(-\beta\mathcal{H}\big(\exp(X)\big)\Big) \cdot \mathbf{1}_{\{ X\in \cA \}} \cdot \prod_{e \in E(Q_L)}\Big|\det\Big(\frac{I - e^{-{\sf ad}_{X_e}}}{{\sf ad}_{X_e}}\Big)\Big| \prod_{e\in E(Q_L)} {\rm d} \,  {\tt Leb} (X_e) \prod_{e\in \partial Q_L } {\rm d}\delta_{\eta_e}(X_e) \, .
    \end{align}
    There are several components in the density~\eqref{eq:density_of_4}, and we analyze each of them separately. First, for any configuration $X\in \mathfrak{g}^{E(Q_L)}$ and boundary conditions $\boldsymbol{\eta}$ we denote by 
    \[
    M = \max\Big\{ \max_{e\in E(Q_L) } \| X_e\| \,  , \,  \max_{e\in \partial Q_L} \| \eta_e \| \Big\} \, ,
    \]
    and note that the event $\cA = {\sf L}(\cE)$, together with Claim~\ref{claim:taylor_expansion_for_log}, implies that $M \le 2\beta^{\kappa-1/2}$. We further denote by $\overline{E}_L = E(Q_L) \cup \partial Q_L$. Claim~\ref{claim:taylor_expansion_for_exponential_of_matrix} and Claim~\ref{claim:identity_minus_product_of_exp_is_quadratic} implies that 
    \begin{align*}
        \mathcal{H}&\big(\exp(X) \big)  \\ &=  \frac{1}{2} \sum_{\substack{p\in P(Q_L) \\ p = \{e_1,e_2,e_3,e_4\}}} \| I - \exp(X_{e_1}) \exp(X_{e_2}) \exp(-X_{e_3}) \exp(-X_{e_4}) \|^2 + \frac{m}{2}\sum_{e\in \overline{E}_L} \| I - \exp(X_e)\|^2 \\ &= \frac{1}{2} \sum_{\substack{p\in P(Q_L) \\ p = \{e_1,e_2,e_3,e_4\}}} \| ({\sf d} X)_p \|^2 + \frac{m}{2}\sum_{e\in \overline{E}_L} \| X_e \|^2 + O\big(L^d M^3\big) \stackrel{\eqref{eq:def_of_gaussian_action}}{=} S(X) + O\big(L^d M^3\big)\, . 
    \end{align*}
    Furthermore, Claim~\ref{claim:jacobian_of_change_of_variables_close_to_one} shows that
    \[
    \prod_{e \in E(Q_L)}\Big|\det\Big(\frac{I - e^{-{\sf ad}_{X_e}}}{{\sf ad}_{X_e}}\Big)\Big| = \Big(1 + O(M)\Big)^{|E(Q_L)|} = 1 + O\big(M L^d\big) \, .
    \]
    Plugging the above into~\eqref{eq:density_of_4} we get that 
    \begin{equation} \label{eq:density_of_4_after_applications_of_claims}
        {\rm d } \, {\sf L}_\ast \bP_{\boldsymbol{\gamma}}^{\sf YMH}(X) \cdot \mathbf{1}_{\cA}(X) \propto e^{-\beta S(X)}  \cdot \mathbf{1}_{\{ X\in \cA \}} \cdot \big(1 + \psi(X)\big) \prod_{e\in E(Q_L)} {\rm d} \,  {\tt Leb} (X_e) \prod_{e\in \partial Q_L } {\rm d}\delta_{\eta_e}(X_e) \, ,
    \end{equation}
    where
    \[
    \delta= \sup_{X\in \cA} | \psi(X) | \lesssim ML^d + \beta M^3 L^d \lesssim L^d \beta^{3\kappa - 1/2}  \, .
    \]
    We see from~\eqref{eq:density_of_4_after_applications_of_claims} that the density of ${\sf L}_\ast \bP_{\boldsymbol{\gamma}}^{\sf YMH}\cdot\mathbf{1}_{\cA}$ almost matches the density of the measure $\bP_{\boldsymbol{\eta^\prime}}^{{\mathfrak
    g}} \cdot 1_{\cA}$. To deal with the error term $\psi(x)$, we see that 
    \begin{multline*}
    \int_{\mathfrak{g}^{\overline{E}_L}} e^{-\beta S(X)}  \cdot \mathbf{1}_{\{ X\in \cA \}} \cdot  |\psi(X)| \prod_{e\in E(Q_L)} {\rm d} \,  {\tt Leb} (X_e) \prod_{e\in \partial Q_L } {\rm d}\delta_{\eta_e}(X_e) \\ \le \delta \int_{\mathfrak{g}^{\overline{E}_L}} e^{-\beta S(X)}  \prod_{e\in E(Q_L)} {\rm d} \,  {\tt Leb} (X_e) \prod_{e\in \partial Q_L } {\rm d}\delta_{\eta_e}(X_e) = \delta Z_{\boldsymbol{\eta}}^{\mathfrak{g}} \, .
    \end{multline*}
    Furthermore, the (Gaussian) partition function is not too sensitive to changing the boundary conditions within the constraint~\eqref{eq:good_boundary_conditions_in_comparison_lemma}, and we have the simple bound
    \[
    |Z_{\boldsymbol{\eta}}^{\mathfrak{g}} - Z_{\boldsymbol{\eta}^\prime}^{\mathfrak{g}}| = Z_{\boldsymbol{\eta}}^{\mathfrak{g}} \, \big|1 - \frac{Z_{\boldsymbol{\eta}^\prime}^{\mathfrak{g}}}{Z_{\boldsymbol{\eta}}^{\mathfrak{g}}}\big| \lesssim Z_{\boldsymbol{\eta}}^{\mathfrak{g}} \, L^{d-1} M  \lesssim Z_{\boldsymbol{\eta}}^{\mathfrak{g}} \, L^{d-1} \beta^{\kappa-1/2} \, .
    \] 
    Claim~\ref{claim:total_variation_bound_between_gibbs_measures} now implies that 
    \[
    d_{\sf TV}\big( {\sf L}_\ast\bP_{\boldsymbol{\gamma}}^{\sf YMH} \cdot \mathbf{1}_{\cE} , \,  \bP_{\boldsymbol{\eta}^\prime}^{{\mathfrak{g}}} \cdot \mathbf{1}_{\cA} \big) \le \big(\delta + L^{d-1} \beta^{\kappa-1/2} \big)\frac{Z_{\boldsymbol{\eta}}^{\mathfrak
    g}}{Z_{\boldsymbol{\eta}}^{\mathfrak{g}}}  \lesssim L^d \beta^{3\kappa-1/2} \, , 
    \]
    as desired. 
\end{proof}
\subsection{Concluding the proof of Proposition~\ref{prop:YMH_is_close_to_Proca}}
It remains to show that the total variation distance between (5) and (6) from Table~\ref{table:table} is small. This will be a simple corollary of the fact that $\cA$ is a typical event under $\bP_{\cA_2}^{\mathfrak{g}}$ , which in turn would follow from Gaussian tail bounds. We conclude the section with this bound and then complete the proof of Proposition~\ref{prop:YMH_is_close_to_Proca}.

\begin{proof}[Proof of Proposition~\ref{prop:YMH_is_close_to_Proca}]

    As we observed just above, by combining Claim~\ref{claim:from_3_to_4} and Lemma~\ref{lemma:from_4_to_5} it remains to prove a bound on the distance between measures (5) and (6) in Table~\ref{table:table}. We know from~\eqref{eq:bound_on_total_variation_distance_of_conditional_vs_original} that $d_{\sf TV}\big(\bP^{{\mathfrak{g}}}_{\mathcal{A}}, \bP^{{\mathfrak{g}}}_{\cA_2}\big) = \bP^{{\mathfrak{g}}}_{\cA_2}(\mathcal{A}_1^c)$ so we now bound the probability on the right hand side. Recall that $\cA = {\sf L}(\cE)$, where $\cE$ is given by~\eqref{eq:def_of_E}. 
    Define the event $\mathcal{B} = \Big\{\forall e\in E(Q_L) \, : \, \| X_e \| \le \frac{1}{2} \beta^{\kappa-1/2}\Big\}$. By Claim \ref{claim:taylor_expansion_for_log}, $\mathcal{B} \subset \mathcal{A}_1$, so $\bP_{\cA_2}^{\mathfrak{g}} (\mathcal{A}_1^c) \leq \bP_{\cA_2}^{\mathfrak{g}}(\mathcal{B}^c)$. Further, using the decomposition $X_e = \sum_{i=1}^n X_{i,e}V_i$ for $V_1, \dots, V_n$ an orthonormal basis of $\mathfrak{g}$, a union bound shows
    \begin{equation*}
        \bP_{\cA_2}^{\mathfrak{g}}\left(||X_e|| > \tfrac{1}{2}\beta^{\kappa-1/2}\right) \leq \sum_{i=1}^n\bP_{\cA_2}^{\mathfrak{g}}\left(|X_{i,e}| > \tfrac{1}{2n}\beta^{\kappa-1/2}\right).
    \end{equation*}
    We now bound the probability on the right-hand side of the above uniformly in $i=1, \dots, n$. This is done via a standard estimates for the conditional mean and variance of the corresponding Gaussian vector, where the conditioning is given by $\cA_2$. This is fairly standard, but for completeness we provide the details for this conditioning in Appendix~\ref{sec:appendix} below. Indeed, Lemma~\ref{lemma:conditioning_expectation_bound} shows that for any edge $e\in E(Q_L)$, we have $$\bE^{\mathfrak{g}}_{\cA_2}|X_{i,e}| \leq C\|{\boldsymbol{\eta}}\|_{\infty} \, .$$ In turn, Claim~\ref{claim:trivial_variance_bound_conditioned_Proca} gives the easy estimate $\Var^{\mathfrak{g}}_{\cA_2}(X_{i,e}) \leq (m\beta)^{-1}$. On the event $\cA_2$,  Claim~\ref{claim:taylor_expansion_for_log} implies that for $\beta$ large enough
    \[
    C \| {\boldsymbol{\eta}}\|_{\infty} \le \frac{1}{4n} \beta^{\kappa/2-1/2} \, .
    \]
    The standard Gaussian tail bound now gives
    \begin{equation*}
        \bP_{\cA_2}^{\mathfrak{g}}(|X_{i,e}| > \tfrac{1}{2n}\beta^{\kappa-1/2}) \lesssim \exp\Big(-c m \beta \cdot \beta^{2\kappa -1}\Big) = \exp\Big(-c m \beta^{2\kappa}\Big)
    \end{equation*}
    where $c= c(n,d)>0$. Union bounding over $e \in E(Q_L)$ we get that
    \begin{equation*}
        \bP_{\cA_2}^{\mathfrak{g}}(\mathcal{B}^c) \leq C L^d \exp\Big(-c m \beta^{2\kappa}\Big) \, ,
    \end{equation*}
    and putting this together with Claim~\ref{claim:from_3_to_4} and Lemma~\ref{lemma:from_4_to_5}, we have 
    \begin{equation*}
        d_{{\sf TV}}\Big(\bP_{\cA_2}^{\mathfrak{g}} , \, {\sf L}_\ast \bP_{\cE_2}^{{\sf YMH}} \Big) \le C L^d \, \frac{\log \beta}{\beta^{2\kappa}} + C L^{d} \beta^{3\kappa - 1/2} + C L^d \exp\Big(-c m \beta^{2\kappa}\Big) \, ,
    \end{equation*}
    which gives the desired bound. 
\end{proof}

\section{Continuum Proca as lattice spacing shrinks}
\label{sec:cont_proca_as_lattice_spacing_shrinks}
\noindent The goal of this section is to show how our main result (Theorem~\ref{thm:main_result}) follows from Proposition~\ref{prop:lattice_Proca_converge_to_cont}, and then also prove the latter. 

\subsection{Proof of Theorem~\ref{thm:main_result}}
The result would follow once we show that 
\begin{equation}\label{eq:proof_of_main_result_what_we_want}
    Z^\eps(F) \xrightarrow{\ \eps\to 0 \ } \mathcal{X}_{{\sf g}, m}(F) \, ,
\end{equation}
in law for any test function $F\in \mathcal{D}$. We first approximate $Z^\eps(F)$ by a finite-dimensional Gibbs measure. Indeed, let $L =  \lfloor\eps^{-1-\kappa}\rfloor$ be as in the statement of Proposition~\ref{prop:lattice_Proca_converge_to_cont} and let $\cE_2$ be the event defined in~\eqref{eq:def_of_E}. Lemma~\ref{lemma:large_values_for_YMH_field_are_rare} implies that 
\begin{equation}
    \label{eq:nu_E_2_close_to_1}
    \nu\big(\cE_2^c\big) \lesssim \eps^{\kappa \cdot C_{d,n} - d(1 + \kappa)} \xrightarrow{\ \eps\to 0 \ } 0 \, ,
\end{equation}
provided that $C_{d,n}$ in~\eqref{eq:assumption_on_rate_of_beta_in_terms_of_eps} is sufficiently large. Therefore,~\eqref{eq:proof_of_main_result_what_we_want} would follow once we show that
\begin{equation} \label{eq:proof_of_main_result_sufficient}
    Z^\eps(F) \cdot \mathbf{1}_{\cE_2} \xrightarrow{\ \eps\to 0 \ } \mathcal{X}_{{\sf g}, m}(F) \, ,
\end{equation}
in law. Using~\eqref{eq:nu_E_2_close_to_1} one more, we conclude from Proposition~\ref{prop:YMH_is_close_to_Proca} that 
\[
d_{\sf TV} \Big( Z^\eps(F) \cdot \mathbf{1}_{\cE_2} \,  , \, Z^{\mathfrak{g}, \eps}(F) \cdot \mathbf{1}_{\cE_2} \Big) \lesssim d_{{\sf TV}}\Big(\bP_{\beta,m,\cA_2}^{\mathfrak{g}} , \, {\sf L}_\ast \bP_{\beta,m,\cE_2}^{{\sf YMH}} \Big) \xrightarrow{\ \eps\to 0 \ } 0 \, .
\]
All in all,~\eqref{eq:proof_of_main_result_sufficient} now follows from the fact that
\[
Z^{{\mathfrak{g}},\eps}(F) \xrightarrow{ \ \eps\to 0 \ } \mathcal{X}_{{\sf g},m}(F)
\]
in law, which is exactly the statement of Proposition~\ref{prop:lattice_Proca_converge_to_cont}. This proves~\eqref{eq:proof_of_main_result_what_we_want} and we are done.
\qed

\subsection{Proof of Proposition~\ref{prop:lattice_Proca_converge_to_cont}}

As we already mentioned, this part follows closely the argument from~\cite[Section~4]{Chatterjee-HiggsMechanism}, and in particular the proof of Theorem~4.6 therein. As such, we will keep the section rather brief and mostly emphasis the modifications needed for our case. 
 \begin{proof}[Proof of Proposition~\ref{prop:lattice_Proca_converge_to_cont}]
    To ease notation, we henceforth drop the dependence on $\eps, G$ and write $Z = Z^{G,\eps}$. By the scaling property of Proca field, we can assume without loss of generality that $m=1$ (see e.g.\ \cite[Lemma 2.5]{Chatterjee-HiggsMechanism}). Letting $F\in \mathfrak{D}$, we need to show that $$Z(F) \xrightarrow{\ \eps\to 0 \ } X_{\mathfrak{g}}(F)$$ in law. Recall that $V_1,\ldots, V_n$ is a basis of $\mathfrak{g}$, and hence both $F$ and the random field $Z$ admits an orthogonal decomposition $$F = \sum_{\ell=1}^n F^\ell V_\ell \, ,\qquad Z = \sum_{\ell =1}^n Z^\ell V_\ell \, .$$  Since all coordinates are also Gaussian, it follows that $$Z(F) = \sum_{\ell=1}^n \big[\sum_{j=1}^d \int_{\mathbb{R}^d} \langle Z^\ell_j(x), F^\ell_j(x)\rangle \, {\rm d} x \big]=:\sum_{\ell=1}^n Z^\ell(F^\ell)$$ is a sum of independent real-valued Gaussian random variables. Similarly, $X_{\mathfrak{g}}(F) = \sum_{\ell=1}^n X^\ell_{\mathfrak{g}}(F^\ell)$ is also a sum of independent real-valued Gaussian random variables. By the Cram\'er-Wold device, it suffices to show that
    \begin{equation}
        \label{eq:proof_lattice_proca_to_cont_after_reduction}
        Z^\ell(F^\ell)\xrightarrow{\eps\to 0} X^\ell_{\mathfrak{g}}(F^\ell)
    \end{equation}
    in law for each $1\le \ell\le n$. The desired convergence~\eqref{eq:proof_lattice_proca_to_cont_after_reduction} follows immediately from~\cite[Theorem~4.6]{Chatterjee-HiggsMechanism}. Still, for completeness and to ease on the readability, we provide a sketch of the argument below. 
    \end{proof}       
    
    \begin{proof}[Sketch of proof of~\eqref{eq:proof_lattice_proca_to_cont_after_reduction}]
    Fix some $\ell\in\{1,\ldots,n\}$. Since $Z^\ell(F^\ell)$ and $X_{\mathfrak{g}}(F^\ell)$ all have a centered Gaussian distribution,~\eqref{eq:proof_lattice_proca_to_cont_after_reduction} amounts to showing that the corresponding variances converge. The basic idea is to compare the differential operator $R_1$ given by~\eqref{eq:def_of_R_m} to the corresponding difference operator induced by the lattice $\eps \bZ^d$. More formally, we let $u\in \bR^{E}$ be defined via
    \[
    u(a, a + e_i):= \int_{\eps D_a} F^\ell_i(q)\, {\rm d} q \, ,
    \]
    where $D_a$ denotes the cube $a+ [-\tfrac{1}{2},\tfrac{1}{2}]^d$. Note that $Z^\ell$ has the law of the Euclidean lattice Proca field (as given in~\cite[Definition~4.5]{Chatterjee-HiggsMechanism}) and let $\{R(e,e^\prime)\}_{e,e^\prime}$ denote the (unscaled) covariance matrix of it. Denoting by $\widetilde R = \eps^{-(d-2)} R$, we have that
    \[
    \mathrm{Var}(Z^\ell(F^\ell)) = u^\top\widetilde{R}u \, .
    \]
    On the other hand, since $$\mathrm{Var}(X^\ell_{\mathfrak{g}}(F^\ell)) = (F^\ell,R_1 F^\ell) = \lim_{\eps \to 0} \sum_{e = (a, a + e_i)} \eps^d G_i(\eps a) F_i(\eps a) \, , $$ where $G^\ell:= R_1 F^\ell \in \mathcal{A}(\bR^d)$, we can approximate $\mathrm{Var}(X^\ell_{\mathfrak{g}}(F^\ell))$ via a sequence of finite differences. Define $x(e):= G^{\ell}_i(\eps a)$ and $w(e):= \eps^d F_i^\ell (\eps a)$, then $\displaystyle \mathrm{Var}(X^\ell_{\mathfrak{g}}(F^\ell)) = \lim_{\eps \to 0} w^\top x$. We expect that 
    $$u(a, a + e_i) \approx w(a, a + e_i) \qquad \text{and} \qquad \tilde{R}u(a, a + e_i) \approx x(a, a + e_i)$$ as $\eps\to 0$ since $\tilde{R}$ is a lattice approximation of $R_1$. However, we need to keep track of the error terms and show that $|u^\top \widetilde{R} u - w^\top x| = o(1)$ as $\eps \to 0$. We claim the following error bounds:
    \begin{itemize}
        \item[(A)] We have $\|u\|, \|w\| \leq O(\eps^{d/2})$ and $\|u-w\| \leq O(\eps^{\frac{d+2}{2}})$;
        \vspace{2mm}
        \item[(B)] We have $\|\widetilde{R}\| \leq O(\eps^{-d/2}),$ and in particular $\|w - \tilde{R}^{-1}x\| \leq O(\eps^{\frac{d+2}{2}})$.
    \end{itemize}
   Assuming (A) and (B), the desired limit follows at once, since by the triangle inequality we have 
    %The result follows from the triangle inequality:
    \begin{align*}
        |u^\top \widetilde{R} u - w^\top x| &\leq |u^\top \widetilde{R} u - w^\top \widetilde{R} w| + |w^\top \widetilde{R} w - w^\top x|\\
        &\leq |u^\top \widetilde{R} (u-w)| + |(u-w)^\top \widetilde{R} w| + |w^\top \widetilde{R} w - w^\top \widetilde{R} \widetilde{R}^{-1} x| \\
        &\leq \|u\|\|\widetilde{R}\| \|u-w\| +  \|u-w\|\|\widetilde{R}\|\|w\| + \|w\| \|\widetilde{R}\| \|w - \widetilde{R}^{-1} x\| \leq O(\eps).
    \end{align*}
    The above display shows that
    \[
    \lim_{\eps \to 0} \text{Var}\big(Z^\ell(F^\ell) \big) = \text{Var}\big(X^\ell_{\mathfrak{g}}(F^\ell)\big)
    \]
    which in turn proves~\eqref{eq:proof_lattice_proca_to_cont_after_reduction}.
    While we do not provide a proof of the estimates (A) and (B) mentioned above (they are derived in the proof of~\cite[Theorem~4.6]{Chatterjee-HiggsMechanism}), we only note that they are not difficult to establish. Indeed, (A) follows by noting that $F$ is a smooth function, which in turn applies that that its variation on the Voronoi cells $D_a$ should be controlled by the volume. The estimate (B) follows from the relation $R_m = (mI - {\sf d} {\sf d}^\ast )^{-1}$ (see~\eqref{eq:def_of_R_m}), and the fact that $\widetilde R$ is simply a lattice approximation of $R_1$, in the sense that the differential operator ${\sf d}{\sf d}^\ast$ is substituted with the corresponding finite differences on the lattice $\eps \bZ^d$ -- see in particular~\cite[Eq. (4.7)]{Chatterjee-HiggsMechanism}. We also note that in the proof of the estimates (A) and (B) one needs to use the fact that off-diagonal elements of the covariance matrix $\widetilde{R}$ decay exponentially in the distance, a fact which we prove in the appendix (see Lemma~\ref{lemma:lattice_Proca_is_massive}). 

  \end{proof}

\bibliographystyle{abbrv}
\bibliography{YangMills}

\pagebreak 

\appendix 
\section{Properties of the Lie-algebra valued Proca Field}
\label{sec:appendix}

In this appendix, we will prove basic properties of the Lie-algebra valued Proca field. The key to many of the proofs is Remark~\ref{remark:cordinates_are_independent_for_free_bdy_conditions}. This remark states that for free boundary conditions, the coordinates of the massive Proca field are i.i.d., and in fact each projection has the distribution of a $d$-dimensional Euclidean lattice Proca field, as defined in~\cite{Chatterjee-HiggsMechanism}. 

Indeed, let $V_1, \dots, V_n$ be an orthonormal basis of $\mathfrak{g}$ with respect to the Hilbert-Schmidt inner product~\eqref{eq:HS_inner_product}, and write $X_e = \sum_{i=1}^n X_{i,e} V_i$. Then if $X \sim \bP_{\beta,m,\text{\normalfont free}}^{\mathfrak{g}}$ is a massive Proca field with free boundary conditions, each $(X_{i,e})_{e \in E(Q_L)}$ is an independent $\bR^d$-valued massive Proca field with free boundary conditions. In particular, each $(X_{i,e})_{e \in E(Q_L)}$ is a Gaussian random vector with density proportional to
\begin{equation*}
    \exp\left( -\frac{1}{2}\sum_{p \in P(Q_L)} ||({\sf d}X_{i})_p||^2 - \frac{m}{2}\sum_{e \in E(Q_L)} X_{i,e}^2\right),
\end{equation*}
where $({\sf d}X_{i})_p = X_{i,e_1} + X_{i,e_2} - X_{i,e_3} - X_{i,e_4}$ with $e_1, e_2, e_3$, and $e_4$ being the edges in the plaquette $p$. In the proofs of the lemmas below, we will often prove the result by just considering the projection to the first coordinate $X_{1,e}$, and then conclude by the aforementioned decomposition of the massive Proca field. By a slight abuse of notation, we will drop the subscript and just write $X(e)$ for $X_{1,e}$. 

\subsection{Decay of correlations}
The first application of this decomposition is to prove that the Proca field is in fact massive.

\begin{lemma} \label{lemma:lattice_Proca_is_massive}
    Denote by $(X_e)_{e\in E(Q_L)}$ a realization of the lattice $\mathfrak{g}$-valued Proca field with free boundary conditions, that is, distributed according to the measure $\bP_{\text{\normalfont free}}^{\mathfrak{g}} = \bP_{\beta,m,\text{\normalfont free}}^{\mathfrak{g}}$ from Definition~\ref{def:lattice_g_valued_proca_field}. Then there exist constants $C,c>0$ (depending only on the lattice dimension $d\ge 2$ and $n=\text{\normalfont dim}(G)$) so that
    \[
    \max_{1\le i,j\le n} \, \bE_{\text{\normalfont free}}^{\mathfrak{g}} \Big[ \langle X_e,V_i\rangle \,  \langle X_{e^\prime} , V_j\rangle  \Big] \le C \beta^{-1} \,  e^{-c \, \min\{m,1\} \, \text{\normalfont dist}(e,e^\prime)},
    \]
    where $\text{\normalfont dist}(\cdot,\cdot)$ is the graph distance on $\bZ^d$.
\end{lemma}

\noindent 
Different versions of this lemma have already appeared in the literature; see for example~\cite[Lemma~12.1]{Balaban-Imbrie-Jaffe-Brydges} or \cite[Lemma~4.10]{Chatterjee-HiggsMechanism}. For the reader's convenience, we provide the proof below. 
\begin{proof}[Proof of Lemma~\ref{lemma:lattice_Proca_is_massive}]
    By scaling, we may assume without loss of generality that $\beta = 1$. Furthermore, by the above considerations it is enough to prove the lemma for $i= j = 1$, and write $X(e) \coloneqq X_{1,e}$ as mentioned above. Then $(X(e))_{e \in E(Q_L)}$ is a multivariate Gaussian with density proportional to
    \begin{equation*}
        \exp\left( -\frac{1}{2}\sum_{p \in P(Q_L)} ||({\sf d}X)_p||^2 - \frac{m}{2}\sum_{e \in E(Q_L)} X(e)^2\right).
    \end{equation*}
    As we are dealing with lattice differential forms (see \cite[Section~10]{Cao-Sheffield-FGF-overview} for a review), we can write 
    \begin{equation*}
        ||({\sf d}X)_p||^2 = \langle({\sf d}X)_p, ({\sf d}X)_p\rangle = \langle X, {\sf d}^* {\sf d}X\rangle,
    \end{equation*}
    where ${\sf d}^*$ is the lattice codifferential (as spelled out in~\cite[Definition 10.10]{Cao-Sheffield-FGF-overview}), and the second equality follows since ${\sf d}$ and ${\sf d}^*$ are adjoint. Thus, if we define the operator $R_m = mI + {\sf d}^* {\sf d}$ on $\bR^{E(Q_L)}$, then $(X(e))_{e \in E(Q_L)}$ is a multivariate Gaussian with density proportional to
    \begin{equation*}
        \exp\left(-\frac{1}{2}\langle X, R_m X\rangle\right).
    \end{equation*}
    We want to show that
    \begin{equation}
        \label{eq:proof_of_lemma:decay_of_correlation_for_lattice_proca_what_we_want}
        \big|(mI + {\sf d}^* {\sf d })^{-1} (e,e^\prime) \big| \le  C e^{-c  \, \text{dist}(e,e^\prime)}
    \end{equation}
    for some $m$-independent constants $C,c>0$. Note that for all $\psi\in \bR^{E(Q_L)}$ we have
    \[
    \langle R_m\psi ,\psi\rangle = m \| \psi \|^2 + \|d\psi\|^2.
    \]
    By Cauchy-Schwarz, there exists a constant $C_d$ (depending only on lattice dimension) so that
    \[
    \|{\sf d}\psi\|^2 \le C_d \|\psi \|^2 \, .
    \]
    Hence, we conclude that the spectrum of $R_m$ must lie in the interval $[m,m+C_d]$. Now set
    \[
    S = I - \frac{1}{m+C_d} R_m ,
    \]
    and note that $S$ is non-negative definite with maximal eigenvalue at most $1-m/(m+C_d) < 1$. We get that
    \[
    (m+C_d) R_m^{-1} = (I - S)^{-1} = \sum_{k = 0}^\infty S^k,
    \]
    where the series on the right-hand side converges absolutely. Since the matrix $S$ has non-zero entries only on a band of constant width around the diagonal, we conclude that
    \begin{align*}
        \big|(mI + {\sf d}^* {\sf d})^{-1} (e,e^\prime) \big| &\le \frac{1}{m+C_d} \sum_{k=0}^\infty \big|S^k(e,e^\prime) \big| \\ & = \frac{1}{m+C_d} \sum_{k=c \, \text{dist}(e,e^\prime)}^\infty \big|S^k(e,e^\prime) \big| \\ &\le \frac{1}{m+C_d} \sum_{k=c \, \text{dist}(e,e^\prime)}^\infty \|S\|^k \\ &\le \frac{1}{m+C_d} \sum_{k=c \, \text{dist}(e,e^\prime)}^\infty \big(1-\frac{m}{m+C_d}\big)^k
    \end{align*}
    which gives~\eqref{eq:proof_of_lemma:decay_of_correlation_for_lattice_proca_what_we_want} and we are done.

\end{proof}

\subsection{Boundary Conditions and Gaussian Conditioning}
The lemmas in this section show that the massive Proca field with boundary conditions is close to the massive Proca field with free boundary conditions. The main reason for this is due to the exponential decay of correlations in Lemma~\ref{lemma:lattice_Proca_is_massive}. Below, let $X \sim \bP_{\beta,m,\boldsymbol{\eta}}^{{\mathfrak{g}}}$ be a massive Proca field with boundary condition $\boldsymbol{\eta}$, and let $Y\sim \bP_{\beta,m,\text{\normalfont free}}^{\mathfrak{g}}$ be a massive Proca field with free boundary conditions. As in Remark~\ref{remark:cordinates_are_independent_for_free_bdy_conditions}, for each edge $e$, we expand
\begin{equation*}
    X_e = \sum_{i=1}^n X_{i,e}V_i, \qquad Y_e = \sum_{i=1}^n Y_{i,e}V_i.
\end{equation*}
We let the \textit{covariance matrix of $X$ (resp. $Y$)} be the covariance matrix of the random variables $(X_{i,e})_{1\leq i \leq n, \,e \in E(Q_L)}$ (resp. $(Y_{i,e})_{1\leq i \leq n, \,e \in E(Q_L)}$). In order to relate the two fields, we use standard multivariate Gaussian conditioning facts.

Again we will consider just the first coordinate $(X(e))_{e \in E(Q_L)} = (X_{1,e})_{e \in E(Q_L)}$ and similarly for $Y$. We denote the boundary condition induced on $(X_{1,e})_{e \in \partial E}$ by $\boldsymbol{\eta}$ as $\boldsymbol \eta_1$. We will write 
\begin{equation*}
    X = \begin{bmatrix}
        X^{\circ}\\
        \partial X \\
    \end{bmatrix}, \qquad \qquad 
    Y =  \begin{bmatrix}
        Y^{\circ}\\
        \partial Y \\
    \end{bmatrix},
\end{equation*}
where $X^{\circ}$ corresponds to the interior edges and $\partial X $ to the boundary edges, and similarly for $Y$. By symmetry, since $Y$ has free boundary conditions, we know that $\bE [Y(e)] = 0$ for all $e \in E(Q_L)$. Thus we can write 
\begin{equation*}
    Y \sim\mathcal{N}(0, R),
\end{equation*}
where $R(e, e') = \bE[Y(e)Y(e')]$. Further, we decompose $R$ as 
\begin{equation*}
    R = \begin{bmatrix}
        T & Q \\
        Q^T & S
    \end{bmatrix}, 
\end{equation*}
where $T$ is the covariance matrix of the interior edges, $S$ is the covariance matrix of the boundary edges, and $Q$ is the covariance matrix of the cross terms. Explicitly, we have
\begin{align*}
   (T(e, e'))_{e, e' \in E^{\circ}} &= (R(e, e'))_{e, e' \in E^{\circ}} \\
   (S(e, e'))_{e, e' \in \partial E} &= (R(e, e'))_{e, e' \in \partial E} \\
   (Q(e, e'))_{e \in E^{\circ}, e' \in \partial E} &= (R(e, e'))_{e \in E^{\circ} e', \in \partial E}.
\end{align*}
Then $X$ is distributed as $Y$ conditioned on the boundary values being $\boldsymbol \eta_1$. Thus, by multivariate Gaussian conditioning facts (e.g. \cite[Proposition 3.13]{EatonMultivariateStatistics1983}), we have in the interior
\begin{equation}\label{eq:proca_boundary_conditioning_dist}
    X^{\circ} \sim \mathcal{N}(QS^{-1}{\boldsymbol \eta_1}, T - QS^{-1}Q^{T}).
\end{equation}

Note that since $QS^{-1}Q^T$ is positive definite and thus for any edge $e$, $(QS^{-1}Q^T)(e, e) \geq 0$, this immediately implies the following claim.

\begin{claim}\label{claim:trivial_variance_bound_conditioned_Proca}
    For any edge $e$ in the interior of $Q_L$ and any $i = 1, \dots, n$, 
\begin{equation*}
    \Var(X_{i,e}) \leq \frac{1}{m\beta}.
\end{equation*}
\end{claim}
While the above lemma simply follows from conditioning reducing the variance, the next lemma gives an explicit bound on the difference in covariances of the massive Proca field with boundary conditions and the massive Proca field with free boundary conditions.

\begin{lemma}
    For $M \in \bN$ with  $M < L$, define $ \Sigma_{M,{\boldsymbol{\eta}}}$ and $\Sigma_{M,\text{\normalfont free}}$ to be the covariance matrices of $X$ and $Y$ (as defined above), respectively, restricted to the edges in the box $[-M,M]^d \subset[-L,L]^d$. Then
    \begin{equation*}
        \left\| \Sigma_{M,{\boldsymbol{\eta}}} - \Sigma_{M,\text{\normalfont free}}\right \| \leq \frac{CL^{d-1}M^{d}}{\beta}e^{-cm(L-M)},
    \end{equation*}
    where $C$ is a constant that only depends on the lattice dimension $d \geq 2$, $m$, and $n = \dim(G)$, and $c$ only depends on $d$ and $n$.
\end{lemma}

\begin{proof}
    First, note that if $\sup_{e\in \partial Q_L} \|\eta_e  \| \le C$ for some constant $C$, then for any $i = 1, \dots, n$,  $\sup_{e\in \partial Q_L} |X_{i,e} | \le C$ by the orthonormality of $V_1, \dots V_n$. Thus the boundary condition $\boldsymbol{\eta}$ induces boundary conditions $\boldsymbol{\eta}_i$ on each of the Gaussian vectors $(X_{i, e})_{e \in E(Q_L)}$ which satisfy $||\boldsymbol{\eta}_i||_{\infty} \leq \sup_{e\in \partial Q_L} \|\eta_e\|$.
    
    We now reduce to the case of just the first coordinate Gaussian, i.e., we prove a bound for the covariance matrices of $(X_{1,e})_{e \in E(Q_L)}$ and $(Y_{1,e})_{e \in E(Q_L)}$, and we will abuse notation below by dropping the coordinate 1 and writing $X(e)$ and $Y(e)$. Let $E(Q_M)$ denote the set of edges in $[-M,M]^d\cap \bZ^d$. We thus want to bound
    \begin{equation*}
        \sum_{e,e'\in E(Q_M)}(\cov(X(e),X(e'))-\cov(Y(e),Y(e')))^2.
    \end{equation*}

   \noindent
    By the decomposition in Equation~\ref{eq:proca_boundary_conditioning_dist} for any pair of edges $e, e'$ in the interior of $Q_L$ (and in particular for any $e,e' \in E(Q_M)$), 
    \begin{equation*}
        \cov(X(e),X(e'))-\cov(Y(e),Y(e')) =  (QS^{-1}Q^{T})(e,e').
    \end{equation*}
    For a given edge $e$, we let $q(e) = (Q(e,e'))_{e' \in \partial Q_L}$ be the vector of covariances with the boundary edges. Then it remains to bound the quantity
    \begin{equation*}
        \sum_{e, e' \in E(Q_M)}(q(e)S^{-1}q(e)^*)^2.
    \end{equation*}

\noindent     Now note that $S$ is a principal submatrix of $R$, so the maximal eigenvalue of $S$ is bounded by that of $R$. In the proof of Lemma~\ref{lemma:lattice_Proca_is_massive}, we showed that the spectrum of $\beta R^{-1}$ (denoted $R_m$ in the proof of Lemma~\ref{lemma:lattice_Proca_is_massive}) is contained in the interval $[m, m+C_d]$ for some constant $C_d$ that only depends on the lattice dimension, so the spectrum of $R$ (and hence the spectrum of $S$) is contained in $\left[\frac{1}{\beta(m+C_d)}, \frac{1}{\beta m}\right]$. Thus the above quantity is bounded by
    \begin{equation*}
        \beta^2 (m+C_d)^2\sum_{e, e' \in E(Q_M)} (||q(e)||^2||q(e')||^2) = \beta^2 (m+C_d)^2\left(\sum_{e\in E(Q_M)} ||q(e)||^2\right)^2.
    \end{equation*}

    For a given edge $e$ that is distance $r$ from the boundary, by Lemma~\ref{lemma:lattice_Proca_is_massive}, 
    \begin{equation*}
        ||q(e)||^2 \leq |\partial Q_L| \sup_{e' \in \partial Q_L} \cov(X_e,X_{e'})^2 \leq C\beta^{-2}L^{d-1}e^{-cmr}.
    \end{equation*}
    Then (with $C$ being a constant that only depends on $d$, $m$, and $n = \dim(G)$ possibly changing between lines),
    \begin{align*}
        \sum_{e \in E(Q_M)} ||q(e)||^2 &\leq \frac{CL^{d-1}}{\beta^2}\sum_{r=L-M}^{L}|\{e \in E(Q_M) : d(e, \partial Q_L) = r\}|e^{-cmr} \\
        &= \frac{CL^{d-1}}{\beta^2}\sum_{r=L-M}^{L}(L-r)^{d-1}e^{-cmr} \\
        &\leq \frac{CL^{d-1}}{\beta^2}M^{d}e^{-cm(L-M)}.
    \end{align*}
    Putting everything together, we get that
    \begin{equation*}
        \sum_{e, e' \in E(Q_M)}(q(e)S^{-1}q(e)^*)^2 \leq \frac{CL^{2(d-1)}M^{2d}}{\beta^2}e^{-2cm(L-M)},
    \end{equation*}
    and taking the square root to get the Hilbert-Schmidt norm gives the result.

\end{proof}

\begin{lemma}\label{lemma:conditioning_expectation_bound}
    As above, let $X$ be a massive Proca field with boundary condition $\boldsymbol{\eta}$, and let $||\boldsymbol{\eta}||_{\infty}=\sup_{e \in \partial Q_L} ||\boldsymbol{\eta}_e||$. Then for any edge $e \in E(Q_L)\setminus \partial Q_L$,
    \begin{equation*}
        \bE[||X_e||] \leq C||\boldsymbol{\eta}||_{\infty}e^{-cmd(e,\partial Q_L)},
    \end{equation*}
    where $C$ and $c$ are constants that only depend on the dimension $d \geq 2$, $n = \dim(G)$, and $m$, and $d(e, \partial Q_L)$ is the graph distance from $e$ to the boundary of $Q_L$. In particular, for any $i=1, \dots, n$,
    \begin{equation*}
        \bE[|X_{i,e}|] \leq C ||\boldsymbol{\eta}||_{\infty}e^{-cmd(e,\partial Q_L)}.
    \end{equation*}
\end{lemma}

\begin{proof}
    Since $||X_e|| \leq n\sup_{i=1, \dots, n} |X_{i,e}|$, we again reduce to the case of $i=1$ and abuse notation by writing $X_{1,e}$ as $X(e)$, and we write $\tilde{\boldsymbol{\eta}}$ for the boundary condition induced on $(X(e))_{e \in \partial Q_L}$ by $\boldsymbol{\eta}$. From Gaussian conditioning, we know that
    \begin{align*}
        |\bE[X(e)]| &= \left |\sum_{e_1, e_2 \in \partial E} Q(e, e_1)(S^{-1})_{e_1, e_2} \tilde{\boldsymbol{\eta}}_{e_2} \right | \\
        &\leq \sum_{e_1 \in \partial E} \left(|Q(e, {e_1})|  \sum_{e_2 \in \partial E}|(S^{-1})_{e_1, e_2} \tilde{\boldsymbol{\eta}}_{e_2}| \right).
    \end{align*}
    Now, note that $R^{-1}$ is given by the Proca density in Definition \ref{def:g_valued_proca_field}, and each row thus contains at most $C$ nonzero entries, where $C$ only depends on the lattice dimension $d$. This is because any given edge in $\bZ^d$ is contained in $2(d-1)$ plaquettes, each with 3 other edges (so we can take, for example, $C = 6(d-1)$). Hence each row of $S^{-1}$ only contains at most $C$ nonzero entries and each entry of $S^{-1}$ is proportional to $\beta$, so the above is bounded by
    \begin{equation*}
        C\beta||\boldsymbol{ \eta}||_{\infty}\sum_{e_1 \in \partial E} |Q(e,{e_1})|,
    \end{equation*}
    where $||\boldsymbol{\eta}||_{\infty} = \sup_{e \in \partial E} ||\eta_e||$. Then, by the exponential decay of correlations of the massive Proca field as in Lemma \ref{lemma:lattice_Proca_is_massive}, we have (with $C$ and $c$ are constants depending only on $d$, $n$, and $m$ possibly changing between lines)
    \begin{align*}
        C\beta||\boldsymbol{\eta}||_{\infty}\sum_{e_1 \in \partial E} |q(e)_{e_1}| &\leq C||\boldsymbol{\eta}||_{\infty} \sum_{k=d(e,\partial Q_L)}^{2dL}\sum_{\substack{e_1 \in \partial E \\ d(e_1, e) = k}}e^{-cmk} \\
        &\leq C||\boldsymbol{\eta}||_{\infty}\sum_{k=d(e, \partial Q_L)}^{2dL} (3k)^de^{-cmk} \\
        &\leq C||\boldsymbol{\eta}||_{\infty}e^{-cmd(e,\partial Q_L)},
    \end{align*}
    where $C$ only depends on the lattice dimension $d$.
\end{proof}

\vspace{2cm}

\end{document}